%% file: main.tex
\documentclass[a4paper,12pt]{amsproc}

\usepackage{mathrsfs}
\usepackage{appendix}
\usepackage{amssymb}
\usepackage{fancyhdr}
\usepackage{charter}
\usepackage{typearea} 
\usepackage{pdfsync}
\usepackage{mathrsfs}
\usepackage{dsfont}
\usepackage{appendix}
\usepackage{amssymb}
\usepackage{fancyhdr}
\usepackage{charter}
\usepackage{typearea}
\usepackage{color}
\usepackage{hyperref}
\usepackage{amsmath,amstext,amsthm,amscd}
\usepackage{mathrsfs}
\usepackage[margin=2.5cm]{geometry}
\usepackage{enumitem}
\usepackage{setspace}
\usepackage{bbm} 
\usepackage{charter}

\numberwithin{equation}{section}
\theoremstyle{plain}
\newtheorem{thm}{Theorem}[section]
\newtheorem{prop}[thm]{Proposition}
\newtheorem*{claim}{Claim}
\newtheorem{cor}[thm]{Corollary}
\newtheorem{lem}[thm]{Lemma}

\theoremstyle{definition}
\newtheorem{defin}{Definition}[section]
\newtheorem{assumption}{Assumption}[section]
\newtheorem{statement}{Statement}[section]

\theoremstyle{remark}
\newtheorem{rmk}{Remark}[section]

\input{newcommand}

\makeatletter
\def\l@subsection{\@tocline{2}{0pt}{2.5pc}{5pc}{}}
\def\l@subsubsection{\@tocline{2}{0pt}{5pc}{7.5pc}{}}
\makeatother

\begin{document}

\title{Spectral Kernels and Holomorphic Morse Inequalities for Sequence of Line Bundles}
\author{Yueh-Lin Chiang}

\begin{spacing}{1.5}
\begin{abstract}
Given a sequence of Hermitian holomorphic line bundles $(L_k,h_k)$ over a complex manifold $M$ which may not be compact, we generalize the scaling method in \cite{me} to study the asymptotic behavior of the Bergman kernels and spectral kernels with respect to the space of global holomorphic sections of $L_k$ with $(0,q)$-forms. We derive the leading term of the Bergman and spectral kernels under the local convergence assumption in the sequence of Chern curvatures $c_1(L_k,h_k)$, inspired by \cite{01}. The manifold $M$ may be non-Kähler and $c_1(L_k,h_k)$ may be negative or degenerate. Moreover, we establish the $L_k$-asymptotic version of Demailly's holomorphic Morse inequalities as an application to compact complex manifolds.
\end{abstract}
\thanks{The paper is a continuation of my master thesis \cite{me} under the supervision of Professor Chin-Yu Hsiao, and the idea comes from the work of Professor Dan Coman, Wen Lu, Xiaonan Ma, and George Marinescu \cite{01}.}
\end{spacing}
\begin{spacing}{1}
\maketitle \tableofcontents
\end{spacing}
\newpage
\input{Chapter1}

\input{Chapter2}

\newpage
\input{Chapter3}

\newpage
\input{reference}
\end{document}

%% file: newcommand.tex
\newcommand{\aaa}[1]{\left\vert#1\right\vert}
\DeclareMathOperator{\Ker}{Ker}
\DeclareMathOperator{\Rang}{Rang}
\DeclareMathOperator{\Dom}{Dom}
\DeclareMathOperator{\End}{End}
\DeclareMathOperator{\Id}{Id}
\DeclareMathOperator{\supp}{supp}
\newcommand{\Tl}{\mathcal{T}}
\newcommand{\N}{\mathbb{N}}
\newcommand{\R}{\mathbb{R}}

\newcommand{\Cs}{\mathbb{C}}
\newcommand{\Cn}{\mathbb{C}^{n}}

\newcommand{\ii}{\sqrt{-1}}
\newcommand{\p}{\partial}
\newcommand{\pb}{\bar{\partial}}
\newcommand{\la}{\langle}
\newcommand{\ra}{\rangle}

\newcommand{\tchi}{\Tilde{\chi}}
\newcommand{\To}{\rightarrow}

\newcommand{\Bsko}{B^{(q)}_{0,s}}

\newcommand{\tpb}{\Tilde{\bar{\partial}}}

\newcommand{\ts}{\Tilde{s}}

\newcommand{\tphi}{\Tilde{\phi}}

\newcommand{\tomega}{\Tilde{\omega}}
\newcommand{\T}{T^{*,(0,q)}\Cn}

\newcommand{\TM}{T^{*,(0,q)}M}

\newcommand{\zb}{\bar{z}}

\newcommand{\wb}{\bar{w}}

%% file: Chapter1.tex
\section{Introduction}

 For a holomorphic Hermitian line bundle $(L,h^L)$ over a Hermitian complex manifold $M$, the asymptotic behavior of the Bergman kernel for high tensor power $L^k:=L^{\otimes k}$ has been extensively studied for a long time (cf.\cite{001},\cite{002},\cite{003},\cite{me},\cite{004},\cite{005},\cite{006},\cite{007}). In \cite{me}, the author adopted a simple scaling method to study the large $k$ behavior of Bergman and spectral kernels of $L^k$ with $(0,q)$-forms and obtain the leading term. In this paper, inspired by the work of Coman, Lu, Ma, and Marinescu \cite{01}, we generalize the method in \cite{me} and consider a more general context, a sequence of line bundles. That is, we must replace the line bundles $\{L^k\}_{k\in\N}$ with a sequence of line bundles $\{L_k\}_{k\in\N}$.    

The scaling method has been used in many different geometric objects. In CR geometry, Hsiao and Zhu \cite{Kohn} established the asymptotic behavior of the heat kernel for the Kohn Laplacian and proved the Morse inequalities of the CR manifolds. Similarly, in real geometry, Chen \cite{jt} employed this technique to study the heat kernel of real manifolds and provided a new proof of the classical Morse inequalities. As mentioned above, the author \cite{me} obtained the semi-classical asymptotic of Bergman and spectral kernels in complex geometry. This technique is relatively simple and does not require complicated analytical tools. For this reason, various geometric objects can be studied using this technique.

In the present paper, we establish the local uniform estimate for the scaled Bergman and spectral kernels in Chapter \ref{Chapter 2}. Moreover, the outcome is also valid for Schwartz kernels with respect to bounded operators of the type defined in (\ref{map}). In Section \ref{section 4.3}-Section \ref{1233333344}, we prove the local convergence of the scaled Bergman and spectral kernels, which is the main result of this paper. In Section \ref{section 3.4}, we offer a more straightforward idea to prove the asymptotic of Bergman kernel by the Heat kernel method under the global large spectral gap condition (cf. Assumption \ref{spectral gap 2} ).

\subsection{Set-up and the main results}
Let $\{(L_k,h_k)\}_{k=1}^{\infty}$ be a sequence of Hermitian holomorphic line bundles over a Hermitian complex manifold $(M,\omega)$ where $\omega$ is a positive Hermitian $(1,1)$-form. For an open set $U\subset M$, denote by $\Omega^{0,q}(U,L_k)$ the space of smooth $L_k$-valued $(0,q)$-forms over $U$ and by $\Omega^{0,q}_c(U,L_k)$ the subspace of $\Omega^{0,q}(U,L_k)$ consisting of elements with compact support in $U$. Suppose $s_k$ is a local holomorphic non-vanishing section of $L_k$, then we can relate $s_k$ to a weight function $\phi_k$ with $|s_k|_{h_k}=e^{-2\phi_k}$. The Chern curvature form $c_1(L_k,h_k)$ is locally given by the $(1,1)$-form: \begin{equation}\label{ Chern curvature}
c_1(L_k,h_k)=-\frac{1}{(2\pi)}\cdot 2\pb\p\phi_k=\frac{1}{\pi}\sum_{i,j=1}^{n}\dfrac{\p^2\phi_k}{\p z^i\p\zb^j}dz^i\wedge d\zb^j.    
\end{equation}
Here, $n$ is the complex dimension of $M$. There is a fibrewise Hermitian inner product $\la\cdot|\cdot\ra_{\omega,h_k}$ on $\TM\otimes L_k$ given by \begin{equation}\label{0406}
    \la \eta_1\otimes s_k|\eta_2\otimes s_k\ra_{\omega,h_k}=\la\eta_1|\eta_2\ra_{\omega}e^{-2\phi_k} \quad \text{where }\, \eta_i\in\Omega^{0,q}(M).
\end{equation}
 We also denote $\la\cdot|\cdot\ra_{\omega,\phi_k}:=\la\cdot|\cdot\ra_{\omega,h_k}$ for convenience. Let $L^2_{\omega,\phi_k}(U,\TM\otimes L_k)$ be the completion of $\Omega^{0,q}_c(U,L_k)$ with respect to the inner product $(\cdot|\cdot)_{\omega,\phi_k}:=\int_U \la\cdot|\cdot\ra_{\omega,\phi_k} dV_{\omega}$. Here, $dV_{\omega}$ is the volume form $\frac{\omega^n}{n!}$. Denote $\|\cdot\|_{\omega,\phi_k,U}$ as the induced norm. Next, the $L_k$-valued Cauchy-Riemann operator is denoted by $\pb^{q}_{k}:\Omega^{0,q}(M,L_k) \To \Omega^{0,q+1}(M,L_k)$, and by $\pb^{q,*}_{k}:\Omega^{0,q}(M,L_k) \To \Omega^{0,q-1}(M,L_k)$ the formal adjoint of $\pb^{q-1}_{k}$ with respect to $(\cdot|\cdot)_{\omega,\phi_k}$. If the manifold is compact, we denote \[
  \mathscr{H}^{0,q}(M,L_k):=\frac{\Ker \pb^q_k:\Omega^{0,q}(M,L_k)\To\Omega^{q+1}(M,L_k)}{\Rang \pb^{q-1}:\Omega^{0,q-1}(M,L_k)\To\Omega^{0,q}(M,L_k)},
 \]which is the Dolbeault cohomology. We now fix a point $p\in M$ and make the assumption inspired by \cite{01}. 
 \begin{assumption}\label{assumption}There exists an open set $D$ containing $p$ such that \[
C_k^{-1}c_1(L_k,h_k)=R+o(1)\quad \text{on }\, D \quad \text{in the}\,\, \mathscr{C}^{\infty}\text{-topology},
\]where $C_k$ is a sequence of real numbers with $C_k\To\infty$. Here, $R$ is a Hermitian $(1,1)$-form, which may not be positive or non-degenerate.
\end{assumption}
Now, we identify the form $R$ with the Hermitian matrix $\hat{R}\in \mathscr{C}^{\infty}(D,\End(T^{(1,0)}M))$ such that for each $U,V\in T^{(1,0)}_zM$, $z\in M$,\begin{equation}\label{1159}
\la\hat{R}(z)U|V\ra_{\omega}=\la R(z),U\wedge \Bar{V}\ra,
\end{equation}where $\la\cdot,\cdot\ra$ is the natural pairing of $T^{*,(1,1)}M$ and $T^{(1,1)}M$. We introduce the main application of this paper and start by the notation of $q$-index set. 
\begin{defin}
    Denote $p\in M(q)$ if $\hat{R}(p)$ is non-degenerate at $p$ and has exactly $q$ negative eigenvalues and $n-q$ positive eigenvalues. Also, we define \[
M(\leq q):=\bigcup_{j=0,\cdots,q}M(j) \subset M.
\]
\end{defin}
\begin{thm}[$L_k$-version  holomorphic Morse inequalities]\label{Morse}Let $(M,\omega)$ be a compact complex manifold and $(L_k,h_k)$ be sequence of Hermitian holomorphic line bundles over $M$. If there exists a sequence $C_k\To\infty$ such that Assumption \ref{assumption} holds for $D=M$, then we have the following asymptotic estimates as $k\To\infty$:
\begin{itemize}
    \item (Weakly Morse inequality)
        \begin{equation}\label{morse1}
            \dim\mathscr{H}^{0,q}(M,L_k)\leq \left(\frac{C_k}{2\pi}\right)^n\int_{M(q)}\big|\frac{R^n}{\omega^n}\big|dV_{\omega}+o((C_k)^n).
        \end{equation}Moreover, if the spectral gap condition ( which will be defined later in Def \ref{spectral gap}) holds on $M$, the equality above holds.
    \item (Strong Morse inequality) 
        \begin{equation}\label{morse2}
        \sum_{j=0}^{q}(-1)^{j}\dim\mathscr{H}^{0,j}(M,L_k)\leq \left(\frac{C_k}{2\pi}\right)^n\int_{M(\leq q)}\left(\frac{R^n}{\omega^n}\right)dV_{\omega}+o((C_k)^n).
        \end{equation}If the spectral gap condition (cf. Def \ref{spectral gap}) holds on $M$, then the equality holds. 
        \item (Asymptotic Riemann-Roch theorem)
        In the case $q=n$, the equality in (\ref{morse2}) holds and hence
        \begin{equation}\label{morse3}
          \sum_{j=0}^{n}(-1)^{j}\dim\mathscr{H}^{0,q}(M,L_k)= \left(\frac{C_k}{2\pi}\right)^n\int_{M}\left(\frac{R^n}{\omega^n}\right)dV_{\omega}+o((C_k)^n).  
        \end{equation}
\end{itemize}Here, note that $\frac{R^n}{\omega^n}(z)=0$ if $R(z)$ is degenerate, and $M(\leq n)=\{z\in M;R(z) \text{ is non-degenerate}\}$.
\end{thm}\newpage
 We now formulate the main results of this paper. Define the Kodaira Laplacian \[
\Box^{q}_k:\Dom\Box^{q}_k\subset L^2_{\omega,\phi_k}(M,\TM\otimes L_k)\To L^2_{\omega,\phi_k}(M,\TM\otimes L_k),
\]which is the Gaffney extension (cf.\cite{Gaffney} or \ref{1154}). For any non-negative constant $c\geq 0$, denote \[P^{q}_{k,c}(z,w)\in\mathscr{C}^{\infty} \left(M\times M,(\TM\otimes L_k)\boxtimes (\TM\otimes L_k)\right)\] as the spectral kernel which is the Schwartz kernel of the spectral projection 
\[
{P}^{q}_{k,c}:=\mathbbm{1}_{[0,c]}\left(\Box^{q}_{k}\right):L^2_{\omega,\phi_k}(M,\TM\otimes L^k)\To \mathscr{E}^{q}_{k,\leq c}:=\text{Rang }\mathbbm{1}_{[0,c]}(\Box^q_k),
\] where $\mathbbm{1}_{[0,c]}\left(\Box^{q}_{k}\right)$ is the functional calculus of the indicator function $\mathbbm{1}_{[0,c]}$ with respect to the Kodaira Laplacain $\Box^q_{k}$ (cf.\cite[Chapter 2]{003} or (\ref{1156})). Here, $(\TM\otimes L_k)\boxtimes(\TM\otimes L_k)$ is the vector bundle over $M\times M$ whose fiber at $(z,w)\in M\times M$ is the space of linear transformation from $T^{*,(0,q)}_wM\otimes L_k\mid_{w}$ to $T^{*,(0,q)}_zM\otimes L_k\mid_{z}$. Moreover, the projection\[
P^q_{k}:=P^q_{k,0}:L^2_{\omega,\phi_k}(M,\TM\otimes L_k)\To \Ker\Box^{q}_k,
\] at the lowest level $c=0$ is called the Bergman projection. The Bergman kernel $P^q_{k}(z,w):=P^q_{k,0}(z,w)$ is the Schwartz kernel of $P^q_{k}$.  From now on, we fix a point $p\in M$ and accept the assumption \ref{assumption}. We can take $D$ as a complex chart centered on $p$ such that
\begin{equation}\label{local chart}
\omega(0)=\ii\sum_{i=1}^{n}dz^i\wedge d\zb^i\quad ;\quad R(0)=\sum_{i=1}^{n}2\lambda_{i,p}dz^i\wedge d\zb^i.
\end{equation} Note that $\{\lambda_{i,p}\}$ are the eigenvalues of $\hat{R}(p)$ defined in (\ref{1159}). In the case $p\in M(q)$, we assume $\lambda_{i,p}<0$ for $i=1,\cdots,q$ and $\lambda_{i,p}>0$ for $i=q+1,\cdots,n$.
Next, we take the sequence of non-vanishing holomorphic sections $s_k$ of $L_k$ over $D$ defined by (\ref{loc s_k}). We can localize the spectral kernel $P^q_{k,c}(z,w)$ in $D\times D$ with respect to $s_k$ by writing $P^q_{k,c}(z,w)$ as  \begin{equation}\label{loc1 Bergman}
P^q_{k,c}(z,w)=P^{q,s}_{k,c}(z,w)s_k(z)\otimes (s_k(w))^*,    
\end{equation}
Here, $P^{q,s}_{k,c}(z,w)$ is an element in 
 $\mathscr{C}^{\infty}(D \times D,\TM \boxtimes \TM )$. Moreover, We denote $P^{q,s}_{k}(z,w):=P^{q,s}_{k,0}(z,w)$ for the Bergman kernel case.  
\begin{defin}[local small spectral gap condition]\label{spectral gap}
    For any $q\in\{0,\cdots,n\}$ and an open set $D \subset M$, we say $D$ has local small spectral gap condition with respect to $\{C_k\}$ if there exist $d\in \N$ and $C >0$ such that for all large enough $k$,
    \begin{equation}\label{spectral gap ineq}
        \|\left(I-P^{q}_{k}\right)u\|^2_{\omega,\phi_k,M}\leq  C ({C_k})^{d}\left( \Box^{q}_{k}u\mid u\right)_{\omega,\phi_k,D} \quad \text{for all }\, u\in\Omega^{0,q}_c(D,L_k).
    \end{equation}
\end{defin}
The main theorems describe the asymptotic behavior of scaled spectral kernels of forms with energy lower than $(C_k)^{-d}$ for some $d\in\N$. First, we state the case $p\notin M(q)$: 
\begin{thm}[main theorem for vanishing case]\label{main theorem}
Let $d\in\N$. If $p\notin M(q)$, then  \begin{equation}\label{main theorem eq de}(C_k)^{-n}P^{q,s}_{k,(C_k)^{-d}}(\frac{z}{\sqrt{C_k}},\frac{w}{\sqrt{C_k}})\To 0, \quad \text{in }\, \mathscr{C}^{\infty}\text{-topology.}\end{equation}
Moreover, for the Bergman kernel case $(C_k)^{-n}P^{q,s}_{k}(z/\sqrt{C_k},w/\sqrt{C_k})$, we also have the property of vanishing.
\end{thm}
Next, we state the main theorem for the case $p\in M(q)$.
\begin{thm}[main theorem for non-vanishing case]\label{main theorem 2}Let $d\in\N$. If $p\in M(q)$, then the scaled kernel $(C_k)^{-n}P^{q,s}_{k,(C_k)^{-d}}(z/{\sqrt{C_k}},w/{\sqrt{C_k}})$ converges to
 \begin{equation*}\label{main theorem eq nonde}
       \dfrac{|\lambda_{1,p} \cdots \lambda_{n,p}|}{\pi^n}\,e^{2(\sum_{i=1}^{q}|\lambda_{i,p}|\zb^i w^i+\sum_{i=q+1}^{n}|\lambda_{i,p}|z^i\wb^i-\sum_{i=1}^{n}|\lambda_{i,p}||w^i|^2})(d\zb^1\wedge \cdots \wedge d\zb^q) \otimes (\dfrac{\p}{\p \wb^{1}} \wedge \cdots \wedge \dfrac{\p}{\p \wb^{q}}),
    \end{equation*}in $\mathscr{C}^{\infty}$-topology. Here, we identify $(d\zb^1\wedge \cdots \wedge d\zb^q) \otimes (\dfrac{\p}{\p \wb^{1}} \wedge \cdots \wedge \dfrac{\p}{\p \wb^{q}})$ as a section of $\T\boxtimes\T$ over $\Cn$ defined by \[\eta \mapsto (d\zb^1\wedge \cdots \wedge d\zb^q)\otimes \eta(\dfrac{\p}{\p \wb^{1}} \wedge \cdots \wedge \dfrac{\p}{\p \wb^{q}}) \quad \text{ for all }\, \eta\in \T.\] Moreover, if the local small spectral gap condition (cf. Def. \ref{spectral gap}) holds, the scaled Bergman kernel $(C_k)^{-n}P^{q,s}_{k}(z/{\sqrt{C_k}},w/{\sqrt{C_k}})$ has the same asymptotic.
    
\end{thm}
 
We next discuss the applications of the main theorems. Note that
\[
P^{q,s}_{k,c}(p,p)=P^{q,s}_{k,c}(z/\sqrt{C_k},w/\sqrt{C_k})\mid_{(z,w)=(0,0)}.\] Since $L_k\otimes (L_k)^*=\Cs$, we can identify $P^q_{k,c}(p,p)$ with an element of $\End(T^{*,(0,q)}_pM)$ and observe that \[
P^q_{k,c}(p,p)=P^{q,s}_{k,c}(p,p).
\]

By the main theorems, 
\begin{align*}
    P^{q}_{k,(C_k)^{-d}}(p,p)&= \left(\frac{C_k}{2\pi}\right)^n\aaa{\frac{R^n}{\omega^n}(p)}\cdot\mathbbm{1}_{ M(q)}(p)(d\zb^1\wedge \cdots \wedge d\zb^q) \otimes (\dfrac{\p}{\p \wb^{1}} \wedge \cdots \wedge \dfrac{\p}{\p \wb^{q}})+o((C_k)^n). 
\end{align*}If spectral gap condition (cf. Def. \ref{spectral gap}) holds, we also have 

 \begin{align*}
    P^{q}_{k}(p,p)&= \left(\frac{C_k}{2\pi}\right)^n\aaa{\frac{R^n}{\omega^n}(p)}\cdot\mathbbm{1}_{ M(q)}(p)(d\zb^1\wedge \cdots \wedge d\zb^q) \otimes (\dfrac{\p}{\p \wb^{1}} \wedge \cdots \wedge \dfrac{\p}{\p \wb^{q}})+o((C_k)^n). 
\end{align*} 
We get an asymptotic of the index density:\begin{equation}\label{0410wow}
\text{Tr}P^{q}_{k,(C_k)^{-d}}(p,p)=\left(\frac{C_k}{2\pi}\right)^n\aaa{\frac{R^n}{\omega^n}(p)}\cdot\mathbbm{1}_{ M(q)}(p)+o((C_k)^n),\quad \text{as}\,\,k\To\infty.
\end{equation} Also, if the spectral gap condition holds,\begin{equation}\label{0410eoeoe}
\text{Tr}P^{q}_{k}(p,p)=\left(\frac{C_k}{2\pi}\right)^n\aaa{\frac{R^n}{\omega^n}(p)}\cdot\mathbbm{1}_{ M(q)}(p)+o((C_k)^n),\quad \text{as}\,\,k\To\infty.
\end{equation} To apply the results to index theory, we need the local uniform bounds:
\begin{thm}[Corollary \ref{0410cor}, Section \ref{Aaaac}]\label{0410thm}For any compact set $K\subset D$, there exists a constant $\Tilde{C}_{K,\ell}$ independent of $k$ such that
\begin{equation}\label{0410}
 \sup_{z\in K}\aaa{(C_k)^{-n}P^{q}_{k,(C_k)^{-N}}(z,z)}\leq \Tilde{C}_{K}.
\end{equation} Also, the result holds for the Bergman kernel case.
\end{thm}
Combining (\ref{0410wow}), (\ref{0410}), and dominated convergence theorem, we have the following result:
\begin{cor}[Local holomorphic Morse inequalities]\label{04102021}Fix $d\in\N$.
Let $(M,\omega)$ be a complex manifold, and $(L_k,h_k)$ be a sequence of Hermitian holomorphic line bundles over $M$. If there exists a sequence $C_k\To\infty$ such that Assumption \ref{assumption} holds on an open set $D$ , then for any compact set $K\subset V$, we have
\[
\int_{K}\text{Tr}P^{q}_{k,(C_k)^{-d}}(z,z)dV_{\omega}(z) = \left(\frac{C_k}{2\pi}\right)^{n}\int_{K}\big|\frac{R^n}{\omega^n}\big|\cdot \mathbbm{1}_{M(q)}dV_{\omega}+o((C_k)^n),\quad \text{as}\,\, k\To \infty.
\]Moreover, the equality holds for $\int_{K}\text{Tr}P^{q}_{k}(z,z)dV_{\omega}(z)$ if the spectral gap condition (Def. \ref{spectral gap}) holds on $D$.
\end{cor} 
\begin{proof}[proof of Theorem \ref{Morse}]
We apply the result to compact complex manifolds. By the Hodge theorem,
\[
\dim \mathscr{H}^{0,q}(M,L_k)=\dim \Ker\Box^q_{k}.
\]We obtain the formula 
\begin{equation}
    \label{dim0}
  \dim \mathscr{H}^{0,q}(M,L_k)=\int_{M}\text{Tr}P^{q}_{k}(z,z)dV_{\omega}(z) 
\end{equation}Moreover, 
\begin{equation}\label{dim1}
\dim \mathscr{H}^{0,q}(M,L_k) \leq \dim \mathscr{E}^q_{k,\leq (C_k)^{-d}}=\int_{M}\text{Tr}P^{q}_{k,(C_k)^{-d}}(z,z)dV_{\omega}(z).     
\end{equation}If the spectral gap condition (Def. \ref{spectral gap}) holds for $D=M$, we observe that \begin{equation}\label{dim2}
\Ker\Box^q_{k}=\mathscr{E}^q_{k,\leq (C_k)^{-d}} \quad \text{for a large enough }d\in\N.    
\end{equation}
This means that the first equality in (\ref{dim1}) holds. To see the weak Morse inequality (\ref{morse1}), we apply Corollary \ref{04102021}, and identities (\ref{dim0}),(\ref{dim1}), and (\ref{dim2}). Next, note that
\[
(-1)^q|\frac{R^n}{\omega^n}(p)|=\frac{R^n}{\omega^n}(p),\quad  \text{for }\, p\in M(q).
\]
To prove the strong Morse inequality (\ref{morse2}), we can use the linear algebra result from Demailly \cite[Lemma 4.2]{De} or \cite[Lemma 3.2.12]{M} and have
\[
\sum_{j=0}^{q}(-1)^{j}\dim\mathscr{H}^{0,j}(M,L_k)\leq \sum_{j=0}^{q}(-1)^j\dim \mathscr{E}^j_{k,\leq (C_k)^{-d}}.
\]
By considering the complex $(\mathscr{E}^{\bullet}_{k,(C_k)^{-d}},\pb_k^{\bullet})$ and combining the identities (\ref{dim1}) and (\ref{dim2}), we obtain the strong holomorphic Morse inequality. Finally, by the fact that\[
0\longrightarrow \mathscr{E}^{0}_{k,\leq (C_k)^{-d}}\big/\Ker \Box^0_k\overset{\pb^0_k}{\longrightarrow}\mathscr{E}^1_{k,\leq (C_k)^{-d}}\big/\Ker \Box^1_k\overset{\pb^1_k}{\longrightarrow} \cdots \overset{\pb^{n-1}_k}{\longrightarrow} \mathscr{E}^n_{k,\leq (C_k)^{-d}}\big/\Ker \Box^n_k\overset{\pb^n_k}{\longrightarrow} 0 
\] is an exact sequence, we can deduce 
\[
\sum_{q=0}^{n}(-1)^q\dim \left(\mathscr{E}^{q}_{k,\leq (C_k)^{-d}}\big/\Ker \Box^q_k\right)=0.
\]Hence,\[
\sum_{j=0}^{q}(-1)^{j}\dim\mathscr{H}^{0,j}(M,L_k)=\sum_{q=0}^{n}(-1)^q\dim \mathscr{E}^{q}_{k,\leq (C_k)^{-d}}=\int_{M}\text{Tr}P^{q}_{k,(C_k)^{-d}}(z,z)dV_{\omega}(z).
\]By the main theorems, we derive the asymptotic Riemann-Roch theorem (\ref{morse3}).
\end{proof}
We refer readers to the book \cite{M} of X. Ma and G. Marinescu for a comprehensive study of Bergman kernel and holomorphic Morse inequalities.

\newpage

%% file: Chapter2.tex
\section{Localization and scaling method}\label{Chapter 2}

\subsection{Notations}
Let $(M,\omega)$ be a complex manifold with $\text{dim}_{\Cs}M=n$ where $\omega$ is a positive Hermitian $(1,1)$-form. $\omega$ induces a fibrewise Hermitian inner product $\la\cdot|\cdot\ra_{\omega}$ on the vector bundle of $(0,q)$-forms, $\TM$. For an open set $U\subset M$, denote by $\mathscr{C}^{\infty}(U)$ the space of smooth functions on $U$ and by $\mathscr{C}^{\infty}_c(U)$ the subspace of $\mathscr{C}^{\infty}(U)$ consisting of elements with compact support in $U$. Let $\Omega^{0,q}(U)$ be the space of smooth $(0,q)$-forms over $U$ and $\Omega^{0,q}_c(U)$ be the subspace of $\Omega^{0,q}(U)$ whose elements have compact support in $U$. Let $L^2_{\omega}(U,\TM)$ be the completion of $\Omega^{0,q}_c(U)$ with respect to the inner product $(\cdot|\cdot)_{\omega,U}:=\int_U\la\cdot|\cdot\ra_{\omega}dV_\omega$, where $dV_{\omega}:= \omega^n/n!$ is the volume form. If $A$ is a bounded linear map from $L^2_{\omega}(U,\TM)$ to itself, denote $\|A\|_{\omega,U}$ as the operator norm. 

For a holomorphic Hermitian line bundle $(L, h^L)$ over $M$, let $s$ be a local holomorphic trivializing section of $L$ over an open subset $U$ of $M$. The Hermitian metric $h^{L}$ locally corresponds to a weight function $\phi:U \To \R$ such that $|s|^2_{h^L}=e^{-2\phi}$. Denote by $\la \cdot|\cdot\ra_{\phi}:=\la \cdot|\cdot\ra_{h^{L}}$ the fibrewise Hermitian inner product $h^{L}$ on $L$ for convenience. Let $\la\cdot|\cdot\ra_{\omega,\phi}$ be the fibrewise Hermitian inner product of $L\otimes \TM$ induced by $h^L$ and $\omega$ ( cf. \ref{0406}). Let $\Omega^{0,q}(U,L)$ be the space of $L$-valued smooth $(0,q)$-forms with domain $U$ and $\Omega^{0,q}_c(U,L)$ be the subspace of $\Omega^{0,q}(U,L)$ whose elements have compact support in $U$. Define $L^2_{\omega,\phi}(U,\TM\otimes L)$ to be the completion of $\Omega^{0,q}_c(U,L)$ with respect to the inner product $(\cdot|\cdot)_{\omega,\phi,U}:=\int_{U}\la\cdot|\cdot\ra_{\omega,\phi}dV_{\omega}$. If $A$ is a bounded linear map from $L^2_{\omega,\phi}(U,\TM\otimes L)$ to itself, denote $\|A\|_{\omega,\phi,U}$ as the operator norm. Sometimes, we may drop $U$ and write $\|\cdot\|_{\omega,\phi}:=\|\cdot\|_{\omega,\phi,M}$ if there is no risk of ambiguity.

\par

For a holomorphic complex chart $\psi:D\subset M\To\psi(D)\subset \Cn$, we locally have the complex coordinate $z=(z^1,\cdots,z^n)$. We denote $\N_0:=\N\cup\{0\}$ and adopt the standard notation $z^{\alpha}$ for multi-index $\alpha=(\alpha_1,\cdots,\alpha_n)\in (\N_0)^n$. We say that a multi-index $I=(i_1,\cdots,i_q)\in(\N_0)^{q}$ is strictly increasing if $1\leq i_1<\cdots <i_q\leq n$ and denote $d\zb^I:=d\zb^{i_1}\wedge\cdots\wedge d\zb^{i_q}$. A $(0,q)$-form $u$ on $M$ can be locally written as
\[
u\mid_{V}=\sideset{}{'}\sum_{|I|=q} u_I(z)d\zb^I,
\]where $\sum'$ means that the summation is performed only over strictly increasing multi-indices. We denote by $dm$ the standard Lebesgue measure on $\Cn$. For $r>0$, we denote $B(r):=\{z\in\Cn;|z|<r\}$ to be the open ball centered at $0\in\Cn$ with radius $r$.

\subsection{Localization}
Recall that we fix a point $p\in M$ and make Assumption \ref{assumption}.
To localize the problem, we take $D$ as a holomorphic complex chart $\psi: D \To\psi(D)\subset \Cn$ such that $\psi(p)=0$ and \begin{equation}\label{loc eq 2}
\omega(0)=\ii\sum_{i=1}^{n}dz^i\wedge d\zb^i\quad ;\quad R(0)=\sum_{i=1}^{n}2\lambda_idz^i\wedge d\zb^i.
\end{equation} Without loss of generality, we may assume that $D$ is pseudoconvex, $\psi(D)$ is convex and identify $D$ with $\psi(D)$ and $p$ with $0$ by abuse of notation. We now construct a sequence of non-vanishing holomorphic sections of $L_k$ over $D$ by the approach in \cite[Lemma 2.1]{01}. Let $\sigma_k$ be the non-vanishing section of $L_k\To D$ that is parallel with respect to the Chern connenction $\nabla^{L_k}$ along the segments $\{t\cdot 0+(1-t)\cdot q\mid t\in[0,1]\}$ for all $q\in D$. Let $\Gamma_{\sigma_k}$ be the connection form of $\nabla^{L_k}$ with respect to $\sigma_{k}$. That is, $\nabla^{L_k}\sigma_k=\Gamma_{\sigma_k}\otimes\sigma_k$. Let $(\nabla^{L_k})^{0,1}$ be the $(0,1)$-part of the Chern connection. We have $(\nabla^{L_k})^{0,1}=\pb$ where $\pb$ is the standard Cauchy-Riemann operator. We see \[
0=\pb\left((\nabla^{L_k})^{0,1}\sigma_k\right)=\left(\pb(\Gamma_{\sigma_k})^{0,1}\right)\otimes\sigma_k-(\Gamma_{\sigma_k})^{0,1}\wedge(\Gamma_{\sigma_k})^{0,1}\otimes \sigma_k\Rightarrow\pb(\Gamma_{\sigma_k})^{0,1}=0.
\]By the $\pb$-Poincaré lemma, there exist functions $f_k\in\mathscr{C}^{\infty}(V)$ such that $\pb f_k=(\Gamma_{\sigma_k})^{0,1}$. Define \[
\Tilde{s_k}:=e^{-f_k}\sigma_k.
\] Note that $\pb \ts_k=e^{-f_k}\left((\Gamma_{\sigma_k})^{0,1}-(\nabla^{L_k})^{0,1}\right)\sigma_k=0$ which means $\ts_k$ are holomorphic sections. Next, we restrict the domain $V$ of $\Tilde{s}_k$ to an open ball and assume $B(1)\subset V$ for convenience. Let $\Tilde{\phi}_k$ be the weight function of $\ts_k$. By multiplying a constant on $\ts_k$, we may assume $\Tilde{\phi}_k(0)=0$ and have the expansion:
\begin{multline*}
    \Tilde{\phi}_k(z)=\sum_{i=1}^n\left(\dfrac{\p \Tilde{\phi}_k}{\p z^i}(0)z^i+\dfrac{\p \Tilde{\phi}_k}{\p \zb^i}(0)\zb^i\right)\\+\dfrac{1}{2!}\sum_{i,j=1}^{n}\left(\dfrac{\p^2 \Tilde{\phi}_k}{\p z^i \p z^j}(0)z^i z^j+\dfrac{\p^2 \Tilde{\phi}_k}{\p z^i \p \zb^j}(0)z^i \zb^j + \dfrac{\p^2 \Tilde{\phi}_k}{\p \zb^i \p \zb^j}(0)\zb^i \zb^j\right)+O_k(|z|^3).
\end{multline*}We denote $O_k$ as the k-dependent big-O notation, which means that the values of the constants in the big-O estimates depend on the parameter $k$.  Set \[F_k(z):= \sum_{i=1}^{n}\dfrac{\p \Tilde{\phi}_k}{\p z^i}(0)z^i+\sum_{i,j=1}^{n}\dfrac{\p^2 \Tilde{\phi}_k}{\p z^i \p z^j}(0)z^i z^j,\] which are holomorphic functions and let
\begin{equation}\label{loc s_k}
s_k:=e^{-F_k}\ts_k.
\end{equation}In this way, we denote by $\phi_k$ the weight functions of $s_k$ and get
\begin{equation*}
\phi_k(z)=\sum_{i,j=1}^{n}\lambda_{k,i,j}z^i\zb^j+O_k(|z|^3),   
\end{equation*}where $\lambda_{k,i,j}:=\dfrac{1}{2!}\dfrac{\p^2 \Tilde{\phi}_k}{\p z^i \p \zb^j}(0)$. By Assumption \ref{assumption}, we observe that \[
\begin{cases}
    \lambda_{k,i,j}=C_k\delta^{i}_{j}\lambda_i+\varepsilon_{k,i,j} \quad \text{where} \,\, \varepsilon_{k,i,j}=o(C_k);\\
    O_k(|z|^3)=C_k O(|z|^3).
\end{cases}
\] Here, $\varepsilon_{k,i,j}$ is a sequence of numbers which satisfy $\varepsilon_{k,i,j}C_k^{-1}\To 0$. We now define\[ \phi_0(z):=\sum_{i=1}^{n}\lambda_i|z^i|^2,\] and then we can write
\begin{equation}\label{loc eq 3}
\phi_k(z)=C_k\phi_0(z)+\sum_{i,j=1}^{n}\varepsilon_{k,i,j}z^i\zb^j+C_kO(|z|^3) \quad \text{where}\,\, \varepsilon_{k,i,j}=o(C_k).
\end{equation}
\subsection{Scaling method}
To begin with, we define the scaled metric $\phi_{(k)}$ and scaled Hermitian form $\omega_{(k)}$ which are defined on $B(\sqrt{C_k})\subset \Cn$ by
\[
\phi_{(k)}(z):=\phi_k(\frac{z}{\sqrt{C_k}})  \quad \text{and} \quad \omega_{(k)}(z):=\omega(\frac{z}{\sqrt{C_k}}),
\]respectively. Denote by $\omega_0$ the standard Hermitian form $\ii\sum_{i=1}^{n}dz^i\wedge d\zb^i$ on $\Cn$. By the identity (\ref{loc eq 3}) and the fact that $\omega(0)=\omega_0(0)$, we have the property: \begin{prop}[convergence of scaled metrics]
\begin{equation}\label{sca con}
\phi_{(k)}\To \phi_0 \quad \text{and} \quad \omega_{(k)}\To \omega_0  \quad \text{in}\,\, \mathscr{C}^{\infty}\text{-topology}.
\end{equation}The convergent rates in (\ref{sca con}) smoothly depend on the point $p$ chosen before, since $\phi$ and $\omega$ are smooth on the manifold $M$.    
\end{prop}

 Inspired by this fact, we define the scaled line bundles \[
L_{(k)} \To B(\sqrt{C_k})
\]which are trivial line bundles with trivializing holomorphic sections \[s_{(k)}(z):= s_k(\frac{z}{\sqrt{C_k}}):B(\sqrt{C_k})\To L_k\simeq L_{(k)}.\]
The scaled line bundle has a scaled metric induced by $h_{(k)}$. That is, \[ \la s_{(k)}(z)|s_{(k)}(z)\ra_{\phi_{(k)}}:=\la s_{(k)}(z)|s_{(k)}(z)\ra_{h_{(k)}}=e^{-2\phi_{(k)}(z)}.\]
We have a fibrewise  inner product $\la\cdot|\cdot\ra_{\omega_{(k)},\phi_{(k)}}$ on the vector bundle $\T\otimes L_{(k)}$ over $B(\sqrt{C_k})$ induced by scaled metric $h_{(k)}$ and scaled Hermitian form $\omega_{(k)}$. That is,
\begin{align*}
\la\eta_1\otimes s_{(k)}|\eta_2\otimes s_{(k)}\ra_{\omega_{(k)},\phi_{(k)}}(z)&=\la \eta_1(z)|\eta_2(z)\ra_{\omega_{(k)}}e^{-2\phi_{(k)}(z)},
\end{align*}for all $\eta_i\in\Omega^{0,q}(B(\sqrt{C_k}))$. By changing variables, we have the unitary identifications:
\begin{align}
L^2_{\omega,\phi_k}(B(1),\TM\otimes L_k)&\cong L^2_{\omega_{(k)},\phi_{(k)}}(B(\sqrt{C_k}),\T\otimes L_{(k)})\,\text{by} \label{sca i1} \\
\eta \otimes s_{k} &\leftrightarrow C_k^{-n/2}\eta(z/\sqrt{C_k})\otimes s_{(k)},\notag\end{align} and \begin{align}
L^2_{\omega}(B(1),\TM)&\cong L^2_{\omega_{(k)}}(B(\sqrt{C_k}),\T)\,\text{by} \label{sca i2}\\
\eta &\leftrightarrow C_k^{-n/2}\eta(z/\sqrt{C_k}).\notag
\end{align}Moreover, there are unitary identifications:
\begin{align}
 L^2_{\omega,\phi_k}(B(1),\TM\otimes L_k)&\cong L^2_{\omega}(B(1),\TM) \,\text{by} \label{sca i3}\\
 \eta\otimes s_k &\leftrightarrow e^{-\phi_k}\eta, \notag 
\end{align} and 
\begin{align}
L^2_{\omega_{(k)},\phi_{(k)}}(B(\sqrt{C_k}),\T\otimes L_{(k)})&\cong L^2_{\omega_{(k)}}(B(\sqrt{C_k}),\T) \, \text{by} \label{sca i4}\\ \eta\otimes s_{(k)} &\leftrightarrow e^{-\phi_{(k)}}\eta.\notag  
\end{align}
The four unitary maps introduced above form a commute diagram.

\subsection{Scaled Laplacians and the elliptic estimates}
The $L_k$-valued Cauchy-Riemann operator is denoted by \[\pb^{q}_{k}:\Dom \pb^q_k \subset L^2_{\omega,\phi_k}(M,\TM\otimes L_k) \To L^2_{\omega,\phi_k}(M,T^{*,(0,q+1)}M\otimes L_k),\] where $\Dom \pb^{q}_k:=\{u\in L^2_{\omega,\phi_k}(M,T^{*,(0,q)}M\otimes L_k);\pb^{q}_ku\in L^2_{\omega,\phi_k}(M,\TM\otimes L_k)\}$. Denote 
\[
\pb^{q,*}_k:\Dom \pb^{q,*}_k\subset L^2_{\omega,\phi_k}(M,\TM\otimes L_k)\To L^2_{\omega,\phi_k}(M, T^{*,(q-1)}\otimes L_k) 
\] as the adjoint of $\pb^{q-1}_{k}$ with respect to $(\cdot|\cdot)_{\omega,\phi_k}$. The Kodaira Laplacian $\Box^q_k$ is defined by\[
\Box^q_k:=\pb_k^{q+1,*}\pb_k^{q}+\pb_k^{q-1}\pb_k^{q,*}:\Dom \Box^q_k \subset L^2_{\omega,\phi_k}(M,\TM\otimes L_k) \To L^2_{\omega,\phi_k}(M,\TM\otimes L_k).
\] This is the Gaffney extension of $\Box^q_k$ where $\Dom \Box^q_k$ is given by (cf.\cite{Gaffney}) 
\begin{equation}\label{1154}
\Dom \Box^{q}_{k}:=\{u\in \Dom \pb^{q}_k \cap \Dom \pb^{q,*}_k\mid \pb^{q}_k u\in \Dom \pb^{q,*}_k\,\text{\rm and }\,\pb^{q,*}_k u\in \Dom \pb^{q-1}_k\}. 
\end{equation} By the identification (\ref{sca i3}), we can translate differential operators acting on sections of $\TM\otimes L_k\mid_{B(1)}$ into operators acting on sections of $\TM\mid_{B(1)}$. Define the localized Cauchy-Riemann operator $\pb^{q}_{k,s}:\Omega^{0,q}(B(1))\To \Omega^{0,q+1}(B(1))$ such that
\[
\pb^{q}_k(\eta\otimes s_k)=e^{\phi_k}\pb^{q}_{k,s}(\eta e^{-\phi_k})\otimes s_k.
\]Let $\pb^{q,*}_{k,s}:\Omega^{0,q}(B(1))\To\Omega^{0,q-1}(B(1))$ be the operator such that\[
\pb^{q,*}_{k}(\eta\otimes s_k)=e^{\phi_k}\pb^{q,*}_{k,s}(\eta e^{-\phi_k})\otimes s_k.
\]Since (\ref{sca i3}) is unitary, we see that $\pb^{q,*}_{k,s}$ is the formal adjoint of $\pb^{q}_{k,s}$ with respect to $(\cdot|\cdot)_{\omega}$. Denote by $\pb^{q}$ the standard Cauchy-Riemann operator acting on smooth sections of $(0,q)$-forms and by $\pb^{q,*}_{\omega}$ the formal adjoint of $\pb^{q}$ with respect to $(\cdot|\cdot)_{\omega}$. We have 
\[
\pb^{q}_{k,s}=\pb^{q}+(\pb\phi_k)\wedge \cdot \quad ; \quad \pb^{q,*}_{k,s}=\pb^{q,*}_{\omega}+((\pb\phi_k)\wedge)^*_{\omega},
\]where $((\pb\phi_k)\wedge)^*_{\omega}$ is the fibrewise adjoint of the wedge operator $(\pb\phi_k)\wedge:\Omega^{0,q-1}\To\Omega^{0,q}$ with respect to $\la\cdot|\cdot\ra_{\omega}$.The localized Kodaira Laplacian is defined by\[
\Box^{q}_{k,s}:= \pb^{q-1}_{k,s}\pb^{q,*}_{k,s}+\pb^{q+1,*}_{k,s}\pb^{q}_{k,s}:\Omega^{0,q}(B(1))\To \Omega^{0,q}(B(1)).
\]  
In the same manner, the $L_{(k)}$-valued Cauchy-Riemann operator is denoted by \[\pb^{q}_{(k)}:\Omega^{0,q}(B(\sqrt{C_k}),L_{(k)})\To\Omega^{0,q+1}(B(\sqrt{C_k}),L_{(k)}),\] which is defined by $\pb^q_{(k)}(\eta\otimes s_{(k)})=(\pb^q\eta)\otimes s_{(k)}$ for all $\eta \in \Omega^{0,q}(B(\sqrt{C_k}))$. Denote \[ \pb^{q,*}_{(k)}:\Omega^{0,q}(B(\sqrt{C_k}),L_{(k)})\To \Omega^{0,q-1}(B(\sqrt{C_k}),L_{(k)})\] to be the formal adjoint of $\pb^{q-1}_{(k)}$ with respect to $(\cdot|\cdot)_{\omega_{(k)},\phi_{(k)}}$. Next, we consider the scaled localized Cauchy-Riemann operator
\[\pb^{q}_{(k),s}:\Omega^{0,q}(B(\sqrt{C_k}))\To\Omega^{0,q+1}(B(\sqrt{C_k})),\] 
which is defined by 
$\pb^{q}_{(k)}(\eta\otimes s_{(k)})=:e^{\phi_{(k)}}\pb^{q}_{(k),s}(\eta e^{-\phi_{(k)}})\otimes s_{(k)}$.  On the other hand, define 
\[\pb^{q,*}_{(k),s}:\Omega^{0,q}(B(\sqrt{C_k}))\To\Omega^{0,q-1}(B(\sqrt{C_k}))\] by $\pb^{q,*}_{(k)}(\eta\otimes s_{(k)})=:e^{\phi_{(k)}}\pb^{q,*}_{(k),s}(\eta e^{-\phi_{(k)}})\otimes s_{(k)}$ which is the formal adjoint of $\pb^q_{(k),s}$ with respect to $(\cdot|\cdot)_{\omega_{(k)}}$. Since (\ref{sca i4}) is unitary, we have similar identities:
\[
\pb^{q}_{(k),s}=\pb^{q}+(\pb\phi_{(k)})\wedge \cdot \quad ; \quad \pb^{q,*}_{(k),s}=\pb^{q,*}_{\omega_{(k)}}+((\pb\phi_{(k)})\wedge)^*_{\omega_{(k)}}.
\]The scaled localized Kodaira Laplacian is given by\[
\Box^{q}_{(k),s}:= \pb^{q-1}_{(k),s}\pb^{q,*}_{(k),s}+\pb^{q+1,*}_{(k),s}\pb^{q}_{(k),s}:\Omega^{0,q}(B(\sqrt{C_k}))\To \Omega^{0,q}(B(\sqrt{C_k})).
\] We have the relations: 
\begin{equation}\label{sca pb}
    (\pb^{q}_{(k),s}u)(\sqrt{C_k}z)=\frac{1}{\sqrt{C_k}}\pb^{q}_{k,s}(u(\sqrt{C_k}z))\quad;\quad(\pb^{q,*}_{(k),s}u)(\sqrt{C_k}z)=\frac{1}{\sqrt{C_k}}\pb^{q,*}_{k,s}(u(\sqrt{C_k}z)).
\end{equation} Hence,
\begin{equation}\label{sca laplacian}
\left(\Box^{q}_{(k),s}u\right)(\sqrt{C_k}z)=\frac{1}{C_k}\Box^{q}_{k,s}(u(\sqrt{C_k}z)). 
\end{equation}
Next, we consider 
 the \textbf{model case} $\Cn$ equipped with the weight function $\phi_0=\sum_{i=1}^n\lambda_i|z^i|^2$ and standard Hermitian form $\omega_0=\ii\sum dz^i\wedge d\zb^i$. The wight function $\phi_0$ defines a Hermitian metric on the trivial line bundle $\Cs\To\Cn$ with $|1|^2_{\phi_0}(z)=e^{-2\phi_0(z)}$. We can define the fiberwise Hermitian metric $\la\cdot|\cdot\ra_{\omega_0,\phi_0}$ on $\T\otimes \Cs\To \Cn$ with respect to $\omega_0$ and $\phi_0$. Let 
\[
\pb^{q}_{0}:\Dom\pb^q_0\subset L^2_{\omega_0,\phi_0}(\Cn,\T\otimes\Cs)\To L^2_{\omega_0,\phi_0}(\Cn,\T\otimes\Cs)
\] be the Cauchy-Riemann operator with values in the trivial line bundle and let $\pb^{q,*}_{0}$ be the formal adjoint of $\pb^{q}_{0}$ with respect to $(\cdot|\cdot)_{\omega_0,\phi_0}$. Denote by \begin{equation}\label{lapl mod 1}
\Box^{q}_{0}:= \pb^{q-1}_{0}\pb^{q,*}_{0}+\pb^{q+1,*}_{0}\pb^{q}_{0}:\Dom\Box^q_0\subset L^2_{\omega_0,\phi_0}(\Cn,\T)\To L^2_{\omega_0,\phi_0}(\Cn,\T)
\end{equation}the Kodaira Laplacian. In the same way, we define the localized Kodaira Laplacian $\Box^{q}_{0,s}$ by \begin{equation}\label{lapl mod 2}
\Box^{q}_{0,s}:=\pb^{q-1}_{0,s}\pb^{q,*}_{0,s}+\pb^{q+1,*}_{0,s}\pb^{q}_{0,s}:\Dom\Box^{q}_{0,s}\subset L^2_{\omega_0}(\Cn,\T)\To L^2_{\omega_0}(\Cn,\T),
\end{equation}where 
\[
\pb^{q}_{0,s}:=\pb^{q}+(\pb\phi_0)\wedge \quad;\quad \pb^{q,*}_{0,s}:=\pb^{q,*}_{\omega_0}+((\pb\phi_0)\wedge)^*_{\omega_0}.
\]
For $m\in\R$ and $u=\sum'_{|I|=q}u_{I}d\zb^I\in \Omega^{0,q}_c(\Cn)$, we adopt the Sobolev norm $\|\cdot\|_{m}$ as
\[
\|u\|^2_{m}:=\sideset{}{'}\sum_{|I|=q}\left(\int_{\Cn}(1+|\xi|^2)^m|\widehat{u}_I(\xi)|^2 dm(\xi)\right),
\] where $\widehat{u}_I(\xi):=(2\pi)^{-n/2}\int _{\Cn}u_I(x)e^{-i\xi \cdot x}dm$. For an open set $U\subset\Cn$, the Sobolev space $W_c^{m}(U,\T)$ is the Banach space given by the completion of $\Omega^{0,q}_c(U)$ with respect to $\|\cdot\|_{m}$. Moreover,  when $m=0$, $\|\cdot\|_{0}$ coincides with the standard $L^2$-norm $\|\cdot\|_{\omega_0}$ induced by the standard Hermitian form $\omega_0$. Hence, we know $W^0_c(U,\T)$ is a Hilbert space with its inner product given by \[
(\cdot|\cdot)_0:=(\cdot|\cdot)_{\omega_0}.
\] By (\ref{sca con}), we deduce that the coefficients of $\Box^{q}_{(k),s}$ converge to the ones of $\Box^{q}_{0,s}$ local uniformly in $\mathscr{C}^{\infty}$-topology and simply write $\Box^q_{(k),s}\To \Box^{q}_{0,s}$. By the elliptic estimate, we have the following Lemma:
\begin{lem}[k-uniform  elliptic estimate]\label{lem elliptic estimate}

   For any bounded domain $U$ and integers $m\in \N$, there is a constant $C(U,p)$ which is continuously dependent on $p$ chosen before and independent of $k$ such that
   \begin{equation*}
       \|u\|_{2m} \leq C(U,p) \left(\|u\|_{0} + \|(\Box^{q}_{(k),s})^{m} u\|_{0}\right)
   \end{equation*}
   for all $u \in W^{2m}_c(U,\T)$ and large enough $k$.
   \end{lem}

\subsection{Scaled kernels and uniform bounds}\label{Aaaac}
To begin with, we start with the spectral theorem:
\begin{thm}\label{uni spe}\cite[theorem 2.5.1]{Davies}
    Let $P:\Dom P\subset \mathcal{H} \rightarrow \mathcal{H} $ be a self-adjoint operator on a Hilbert space $\mathcal{H}$. Then there exists a spectrum set $\text{\rm Spec}P \subset \R$ , a finite measure $\mu$  on $\text{\rm Spec}P \times \N$ and a unitary operator \begin{equation*} H:\mathcal{H} \rightarrow L^2_{d \mu}(\text{\rm Spec}P \times \N)
    \end{equation*}
    with the following properties: Set $h:\text{\rm Spec}\,P \times \N \rightarrow \R$ by $h(s,n):=s$. Then an element $f \in \mathcal{H}$ is in $\Dom P$ if and only if $h \cdot H(f) \in L^2(\text{\rm Spec}P \times \N,d \mu)$. In addition, we have \begin{equation*} Pf = H^{-1}\circ (h \cdot Hf) \quad \text{\rm for all }\, f \in \Dom\,P. 
    \end{equation*}
\end{thm}
Based on Theorem \ref{uni spe}, we know that $\Box_k^{q}$ has the spectrum set $\text{\rm Spec}\,\Box_k^{q}\subset [0,\infty)$ and there is a unitary map \[H^{q}_k:L^2_{\omega,\phi_k}(M,\TM\otimes L_k) \rightarrow L^2_{d \mu_k}(\text{\rm Spec }\Box_k^{q} \times \N)\] such that \[\Box_k^{q}u=(H^{q}_k)^{-1} \circ (h \cdot H^{q}_ku),\] for all $u \in \Dom \Box_k^{q}$. From now on, we let $a^q_k:[0,\infty)\To\R$ be a sequence of real-valued functions such that \[
a^q_k(s,n):=a^q_k(s)\in L^2_{d \mu_k}(\text{Spec} \Box^q_k\times \N).
\]Define 
\begin{equation}\label{map}
A^{q}_k:L^2_{\omega,\phi_k}(M,\TM\otimes L_k)\To L^2_{\omega,\phi_k}(M,\TM\otimes L_k)    
\end{equation}
to be the sequence of self-adjoint bounded linear maps defined by \begin{equation}\label{1156}
A^q_k:=(H_k)^{-1}\circ \left(a^q_k(s,n)\cdot H^q_k\right).
\end{equation}We call $A^q_k$ the functional calculus of $a^q_k$ and write $A^q_k=a^q_k(\Box^q_k)$. Define 
\[A^{q}_k(z,w)\in\mathscr{C}^{\infty}\left(M\times M;(\TM\otimes L_k)\boxtimes(\TM\otimes L_k)^*\right)\] to be Schwartz kernels of $A^q_k$, which we assume to be smooth. For such $A^q_k$, we scale the smooth kernels $A^q_k(z,w)$ in the open ball $B(1)$. First, define the localized kernel $A^{q,s}_{k}(z,w)\in \mathscr{C}^{\infty}(B(1)\times B(1),\TM\boxtimes \TM)$ by
\[
A^{q}_{k}(z,w)=:A^{q,s}_k(z,w)  s_k(z) \otimes (s_k(w))^*.
\]The \textbf{scaled kernels} $A^{q,s}_{(k)}(z,w)\in\mathscr{C}^{\infty}(B(\sqrt{C_k})\times B(\sqrt{C_k}),\T\boxtimes\T)$ are defined by \begin{equation*}
A^{q,s}_{(k)}(z,w):=C_k^{-n}A^{q,s}_k(\frac{z}{\sqrt{C_k}},\frac{w}{\sqrt{C_k}}).    
\end{equation*}
By (\ref{sca i3}),  we define the \textbf{localized map}$A^{q}_{k,s}:L^2_{\omega}(B(1),\TM)\To L^2_{\omega}(B(1),\TM)$ by \[\left(A^{q}_{k,s}(e^{-\phi_k}\eta)\right)\otimes s_k=e^{-\phi_k}A^{q}_k(\eta\otimes s_k).\] Also, by (\ref{sca i2}), we set the \textbf{scaled localized map} $A^{q}_{(k),s}$ acting on $L^2_{\omega_{(k)}}(B(\sqrt{C_k}),\T)$ by \begin{equation}\label{040701}
 \left(A^{q}_{(k),s}\eta\right)(\sqrt{C_k}z)=A^{q}_{k,s}(\eta(\sqrt{C_k}z)).   
\end{equation}Denote by $A^q_{k,s}(z,w)$ and  $A^{q}_{(k),s}(z,w)$ as \textbf{localized kernel} and \textbf{scaled localized kernel} which are the Schwartz kernels of $A^q_{k,s}$ and $A^{q}_{(k),s}$, respectively. The relations between the kernels introduced above are given by \begin{equation}\label{0407}
A^{q}_{k,s}(z,w)=e^{-\phi_k(z)}A^{q,s}_{k}(z,w)e^{\phi_k(w)} \quad ; \quad A^{q}_{(k),s}(z,w)=e^{-\phi_{(k)}(z)}A^{q,s}_{(k)}(z,w)e^{\phi_{(k)}(w)},    
\end{equation}
 and\[
A^q_{(k),s}(z,w)=C_k^{-n}A^q_{k,s}(\frac{z}{\sqrt{C_k}},\frac{w}{\sqrt{C_k}}).
\] For any $u\in\Omega^{0,q}_c(B(\sqrt{C_k}))$, we denote $u_k:={C_k}^{n/2}u(\sqrt{C_k}z)$ and observe the fact that $u_k\in\Omega^{0,q}_c(B(1))\subset \Omega^{0,q}_c(M)$. By changing variable,
\begin{align*}
\|A^q_{(k),s}u\|_{\omega_{(k)},B(\sqrt{C_k})}=\|A^q_{k,s}u_k\|_{\omega,B(1)}&\leq \|A^q_{k} (e^{\phi_k}u_k\otimes s_k)\|_{\omega,\phi_k,M}\leq \|A^q_k\|_{\omega,\phi_k} \|e^{\phi_k}u_k\otimes s_k\|_{\omega,\phi_k}\\
&=\|A^q_k\|_{\omega,\phi_k}\|u_k\|_{\omega,B(1)}=\|A^q_k\|_{\omega,\phi_k}\|u\|_{\omega_{(k)},B(\sqrt{C_k})},
\end{align*}where $\|A^q_k\|_{\omega,\phi_k}$ is the operator norm of $A^q_k$. This computation tells us that the operator norms of $A^q_{(k),s}$, $A^q_{k,s}$ and $A^q_{k}$ have the following relations:
\begin{equation}\label{uni 1}
\|A^q_{(k),s}\|_{\omega_{(k)},B(\sqrt{C_k})}=\|A^q_{k,s}\|_{\omega,B(1)}\leq \|A^q_{k}\|_{\omega,\phi_k,M}.
\end{equation}
Next, since $\omega$ is positive and smooth on the compact set $\overline{D}$, there exist positive constants $C_1(p)$ and $C_2(p)$ that continuously depend on the point $p$ chosen before such that 
\begin{equation}\label{compatible}
C_1(p)\|u\|_{\omega_0,B(\sqrt{C_k})}\leq \|u\|_{\omega_{(k)},B(\sqrt{C_k})}\leq C_2(p)\|u\|_{\omega_0,B(\sqrt{C_k})},
\end{equation}
for all $u\in\Omega^{0,q}(B(\sqrt{C_k}))$ and $k\in\N$. Hence, we can treat \[A^{q}_{(k),s}:L^2_{\omega_0}(B(\sqrt{C_k}),\T)\To L^2_{\omega_0}(B(\sqrt{C_k}),\T)\] a bounded map with the following estimate: \begin{equation}\label{uni 3}
C_1(p)\|A^q_{(k),s}\|_{\omega_0,B(\sqrt{C_k})}\leq \|A^q_{(k),s}\|_{\omega_{(k)},B(\sqrt{C_k})}\leq C_2(p) \|A^q_{(k),s}\|_{\omega_0,B(\sqrt{C_k})}. 
\end{equation}Moreover, by (\ref{sca con}), for any bounded domain $U\subset \Cn$ and  $\varepsilon>0$, there exists $k_0(p)\in\N$ depending on $p$ and is locally constant with respect to $p$ such that
\begin{equation}\label{uni 3.5}
    (1-\varepsilon)\aaa{(u,v)_{\omega_0,U}}\leq \aaa{(u,v)_{\omega_{(k)},U}}\leq (1+\varepsilon)\aaa{(u,v)_{\omega_0,U}},
\end{equation}for all $u,v\in \Omega^{0,q}_c(U)$ and $k>k_0(p)$. Hence, we also have
\begin{equation}\label{uni 4}
    (1-\varepsilon)\|A^q_{(k),s}\|_{\omega_0,U}\leq \|A^q_{(k),s}\|_{\omega_{(k)},U}\leq (1+\varepsilon)\|A^{q}_{(k),s}\|_{\omega_0,U},
\end{equation}for all $k>k_0(p)$. We now extend our estimation of the operator norm of $A^q_{(k),s}$ to the context of Sobolev space. For $m,k\in\N$, we define the number $N_{m,k}$ by\[
N_{m,k}:=\sup_{s\in [0,\infty)}(\dfrac{s}{C_k})^{m} a^q_k(s), 
\]where $A^q_k=a^q_k(\Box^q_k)$. Lemma \ref{lemma mapping property} and Theorem \ref{thm uniform bound} introduced later are adapted from \cite[Theorem 3.4, Theorem 3.5]{me}.

\begin{lem}[k-dependent smoothing property]\label{lemma mapping property}
Fix $\chi$ and $\rho$ in $\mathscr{C}^{\infty}_c(\Cn)$ and an integer $m\in\N$. There exists a constant $C(\chi,\rho,m,p)$ that continuously depends on $p$ chosen before such that\[
\|\chi A^{q}_{(k),s}\rho u\|_{2m}\leq C(\chi,\rho,m,p) \left(\|A^q_k\|_{\omega,\phi_k}+N_{m,k}+N_{2m,k}\right)\|u\|_{-2m},
\]for all $u\in W^{-2m}_c(\Cn,\T)$ and $k\in\N$ with $\supp \chi\cup\supp \rho\subset B(\sqrt{C_k})$. 
\end{lem}
\begin{proof}
We assume $u\in\Omega^{0,q}_c(\Cn)$ by density argument. We choose $U$ as a bounded open domain such that $\supp \chi \cup \supp\rho\subset U$. By Lemma \ref{lem elliptic estimate} and (\ref{compatible}), \begin{align}
\|\chi A^q_{(k),s}&\rho u\|_{2m}\leq C(U,m,p)\left(\|\Tilde{\chi}A^q_{(k),s}\rho u\|_0+\|\Tilde{\chi}(\Box^q_{(k),s})^m A^{q}_{(k),s}\rho u\|_0\right) \label{pf 1}\\
&\leq \Tilde{C}(U,m,p)\left(\|\Tilde{\chi}A^q_{(k),s}\rho u\|_{\omega_{(k)},B(\sqrt{C_k})}+\|\Tilde{\chi}(\Box^q_{(k),s})^m A^{q}_{(k),s}\rho u\|_{\omega_{(k)},B(\sqrt{C_k})}\right),\notag
\end{align}where $\Tilde{\chi}$ is a cut-off function with $\supp\chi \subset\supp \Tilde{\chi}\subset U$. By (\ref{uni 1}) and (\ref{compatible}), 
\begin{equation}\label{041504}
\|\Tilde{\chi}A^q_{(k),s}\rho u\|_{\omega_{(k)},B(\sqrt{C_k})}\leq \|A^q_{k}\|_{\omega,\phi_k,M}\|\rho u\|_{\omega_{(k)}}\leq C(p) \|A^q_{k}\|_{\omega,\phi_k,M}\|u\|_{0}.    
\end{equation}
By (\ref{sca laplacian}) and (\ref{040701}), we have 
\begin{equation}\label{pf 2}
\left(C_k^{n/2}(\Box^q_{(k),s})^mA^q_{(k),s}\rho u\right)(\sqrt{C_k}z)=C_k^{-m}\left(\Box^q_{k,s}\right)^m A^q_{k,s}\rho_k u_k(z).
\end{equation}where $u_k:=C_k^{n/2}u(\sqrt{C_k}z)$ and $\rho_k:=\rho(\sqrt{C_k}z)$. By changing variables, we get
\begin{equation}\label{pf 3}
\|\left(\Box^q_{(k),s}\right)^mA^q_{(k),s}\rho u(z)\|_{\omega_{(k)},B(\sqrt{C_k})}=\|C_k^{-m}\left(\Box^q_{k,s}\right)^m A^q_{k,s}\rho_k u_k\|_{\omega,B(1)}.
\end{equation}By the unitary identification (\ref{sca i3}) and the fact that $A^q_k=a^q_k(\Box^q_k)$,
\begin{align}\label{pf 4}
\|C_k^{-m}\left(\Box^q_{k,s}\right)^m &A^q_{k,s}\rho_k u_k\|_{\omega,B(1)}\leq \|C_k^{-m}\left(\Box^q_{k}\right)^m A^q_{k}\left(\rho_k e^{\phi_k} u_k\otimes s_k\right)\|_{\omega,\phi_k,M}\\
&\leq N_{m,k}\|\rho_k e^{\phi_k} u_k\otimes s_k\|_{\omega,\phi_k,M}= N_{m,k}\|\rho_ku_k\|_{\omega,B(1)}.\notag
\end{align}By changing variables again (or the identity (\ref{sca i3})) and (\ref{compatible}), we have
\begin{equation}\label{eq 0415}
\|\rho_ku_k\|_{\omega,B(1)}=\|\rho u\|_{\omega_{(k)}}\leq C(p)\|u\|_{0}. 
\end{equation}
Combining (\ref{pf 1})-(\ref{eq 0415}), we have another constant $C(U,m,p)$ such that
\begin{equation}\label{pf 5}
    \|\chi A^q_{(k),s}\rho u\|_{2m} \leq C(U,m,p) (\|A^q_k\|_{\omega,\phi_k}+N_{m,k})\|u\|_{0}.
\end{equation}

Now, our goal is to dominate the right-hand side of (\ref{pf 1}) by $\|u\|_{-2m}$. It remains to prove the following claim:

\begin{claim} There exists $C(U,m,p)>0$ which  continuously depends on $p$ such that
\begin{align}
\|\Tilde{\chi}A^q_{(k),s}\rho u\|_0 &\leq C(U,m,p)(\|A^q_k\|_{\omega,\phi_k}+N_{m,k})\|u\|_{-2m};\label{good1}\\
\|\Tilde{\chi}(\Box^q_{(k),s})^m A^{q}_{(k),s}\rho u\|_0&\leq C(U,m,p)(\|A^q_{k}\|_{\omega,\phi_k}+N_{2m,k})\|u\|_{-2m}.
\label{good2}  
\end{align}    
\end{claim}
To prove the claim, we take advantage of the duality property of Sobolev spaces. Let $v\in\Omega^{0,q}_c(\Cn)$ and denote  $v_k(z):=C_k^{n/2}v(\sqrt{C_k}z)$, $\Tilde{\chi}_k(z):=\Tilde{\chi}(\sqrt{C_k}z)$ and $\rho_k:=\rho(\sqrt{C_k}z)$.  By the unitary identifications (\ref{sca i2}), (\ref{sca i3}), and the self-adjointness of $A^q_k$, we have 
\begin{align}\label{pf 6}
    \left(\Tilde{\chi}A^q_{(k),s}\rho u|v\right)_{\omega_{(k)},U}&=\left(\Tilde{\chi}_kA^q_{k,s}\rho_k e^{\phi_k}u_k\otimes s_k|e^{\phi_k} v_k\otimes s_k\right)_{\omega_{k},\phi_k,M}\\
    &=\left(e^{\phi_k}u_k\otimes s_k|\rho_k A^q_{k}\Tilde{\chi}_ke^{\phi_k}v_k\otimes s_k\right)_{\omega,\phi_k,M}=\left(u|\rho A^q_{(k),s}\Tilde{\chi}v\right)_{\omega_{(k)},U}.\notag
\end{align}To estimate the right-hand side above, by (\ref{uni 3.5}), we see
\begin{equation}\label{good3}
 \aaa{\left(u|\rho A^q_{(k),s}\Tilde{\chi}v\right)_{\omega_{(k)}}}\lesssim \aaa{\left(u|\rho A^q_{(k),s}\Tilde{\chi}v\right)_{0}}.     
\end{equation}
Next, we use the duality of Sobolev space and write
\begin{equation}\label{pf 7}
 \aaa{\left(u|\rho A^q_{(k),s}\Tilde{\chi}v\right)_{0}}\leq \|u\|_{-2m}\|\rho A^q_{(k),s}\Tilde{\chi}v\|_{2m}\lesssim (\|A^q_k\|_{\omega,\phi_k}+N_{m,k})\|u\|_{-2m}\|v\|_{0},    
\end{equation}where the last inequality is from (\ref{pf 5}). By (\ref{pf 6})-(\ref{pf 7}) and the fact that $v$ is arbitrary, we have
\begin{equation*}\label{pf 8}
 \|\Tilde{\chi}A^q_{(k),s}\rho u\|_{0}\lesssim (\|A^q_k\|_{\omega,\phi_k}+N_{m,k})\|u\|_{-2m}.    
\end{equation*} All the estimates $\lesssim$ above continuously depend on $p$ and also depend on $U$ and $m$. This proves the inequality (\ref{good1}). For the proof of (\ref{good2}), we adopt the same way and fix a test section $v\in\Omega^{0,q}_c(\Cn)$ again. Then, we mimic the process (\ref{pf 6}) above to get the following estimate:  
\begin{equation}\label{pf 9}
    \left(\Tilde{\chi}(\Box^q_{(k),s})^m A^{q}_{(k),s}\rho u|v\right)_0\lesssim\left(u|\rho(\Box^q_{(k),s})^m A^{q}_{(k),s}\Tilde{\chi}v\right)_0
\end{equation}
 By duality of Sobolev space,  
\begin{equation}\label{pf 10}
 \mid\left(u|\rho(\Box^q_{(k),s})^m A^{q}_{(k),s}\Tilde{\chi}v\right)_0\mid\lesssim \|u\|_{-2m}\|\rho A^{q}_{(k),s}(\Box^q_{(k),s})^m\Tilde{\chi}v\|_{2m}.   
\end{equation}
By repeating the process (\ref{pf 1})-(\ref{pf 5}), we have
\begin{align}\label{pf 11}
 \|\rho A^{q}_{(k),s}(\Box^q_{(k),s})^m\Tilde{\chi}v\|_{2m}&\lesssim \|\Tilde{\rho} A^{q}_{(k),s}(\Box^q_{(k),s})^m\Tilde{\chi}v\|_0+\|\Tilde{\rho}(\Box^q_{(k),s})^m A^{q}_{(k),s}(\Box^q_{(k),s})^m\Tilde{\chi}v\|_0 \\
 &\lesssim (\|A^q_{k}\|_{\omega,\phi_k}+N_{2m,k})\|v\|_{0}\notag
\end{align}We combine (\ref{pf 9})-(\ref{pf 11}) and get
 \begin{equation}\label{pf 12}
    \|\Tilde{\chi}(\Box^q_{(k),s})^m A^{q}_{(k),s}\rho u\|_0 \lesssim (\|A^q_{k}\|_{\omega,\phi_k}+N_{2m,k})\|u\|_{-2m},
 \end{equation}since $v$ is arbitrary. This proves the inequality (\ref{good2}). By (\ref{pf 1}), (\ref{good1}) and (\ref{good2}), the theorem follows. Note that all the estimates above originate from the local behavior of $\omega$ and $\phi$ on manifold $M$. We can see that the estimate continuously depends on the point $p$ chosen at the beginning. 
\end{proof}

Next, we represent $A^{(q)}_{(k),s}(z,w)$ as the form:\[
A^{(q)}_{(k),s}(z,w)=\sideset{}{'}\sum_{|I|=|J|=q} A^{q,I,J}_{(k),s}(z,w)d\zb^I\otimes (\frac{\p}{\p \Bar{w}})^J,
\] where $A^{q,I,J}_{(k),s}(z,w)\in\mathscr{C}^{\infty}(B(\sqrt{C_k})\times B(\sqrt{C_k}))$.

\begin{thm}(The local estimate)\label{thm uniform bound}
 For any bounded domain $U \subset \Cn$, $\ell\in\N$ and strictly increasing multi-indices $I,J\in(\N_0)^{q}$, there exists $C(\ell,U,p)$ such that 
 \[
\aaa{A^{q,I,J}_{(k),s}(z,w)}_{\mathscr{C}^{\ell}(U\times U)}\leq C(\ell,U,p)\left(\|A^q_k\|_{\omega,\phi_k}+N_{m,k}+N_{2m,k}\right),
 \]for all $m\in N$ with $2m\geq \ell+n$ and $k\in\N$. Here, $C(\ell,U,p)$ continuously depends on the point $p$ chosen at the beginning. Here, $|\cdot|_{\mathscr{C}^{\ell}(U\times U)}$ is the usual $\mathscr{C}^{\ell}$-norm with domain $U\times U$.
\end{thm}
\begin{proof}
Denote $x$ and $y$ as the underlying real coordinates of the complex coordinates $z$ and $w$ of $\Cn\simeq \R^{2n}$. Let $\alpha,\beta\in(\N_0)^{q}$ be the multi-indices such that $|\alpha|+|\beta|\leq \ell$.
We start from the approximation of identity. For any fixed point $y_0 \in U$, we set $f_l$ as an approximation of identity with its mass concentrated at $y_0$ as $l \rightarrow \infty$. For example, let $f_l =l^nf(\sqrt{l}(y-y_0))$ where $f \in \mathscr{C}^{\infty}_c(U;[0,\infty))$ and $\int_{U}f dm=1$. By the property of approximation of identity, it is sufficient to dominate the following:
\begin{equation*}
    \sup_{x\in U, l\in \N}| \int_{U}\p^\alpha_x \p^\beta_y A^{q,I,J}_{(k),s}(x,y)f_l(y) dm(y)|.
\end{equation*} We aim to find an estimate independent of $k$ and the point $y_0\in U$ chosen above. By integration by part, we only need to consider
\begin{equation*}
    \sup_{x\in U, l\in \N}|\p^\alpha_x  \int_{U}A^{q,I,J}_{(k),s}(x,y)\p^\beta_y f_l(y)dm(y)|.
\end{equation*}
Choose $\chi \in \mathscr{C}^{\infty}_{c}(\Tilde{U})$ with $U\subset \Tilde{U}$ and $\chi\mid_{U}\equiv 1$. By Sobolev inequality, since $2m\geq |\alpha|+1$,
\begin{align*}
    \sup_{x\in \Tilde{U},l\in \N}|\chi\p^\alpha_x  \int_{U}A^{q,I,J}_{(k),s}(x,y)(\p^\beta_y f_l(y))dm(y)| & \leq \sup_{l\in\N}\|\chi A^{q}_{(k),s}\chi\left((\p^{\beta} f_l)d\zb^J\right) \|_{m}.
\end{align*}
Note that $|\hat{f}_l(\xi)|\lesssim|\int_{\R^{2n}}e^{-\ii x\cdot \xi}f_l(x)dm(x)|= O(1)$ and hence $|\widehat{(\p^{\beta}f_l)}|\lesssim |\xi|^{|\beta|}|\hat{f_l}|\lesssim |\xi|^{|\beta|}$. Since $2m\geq |\beta|+n$, we have
\[
\| (\p^{\beta} f_l)d\zb^J \|_{-2m}\lesssim \int_{\R^{2n}}(1+|\xi|^2)^{-m}|\xi|^{|\beta|}dm=O(1).
\]
After combining this fact with Lemma \ref{lemma mapping property}, we know that 
\begin{align*}
    \| \chi A^{q}_{(k),s}\chi (\p^{\beta} f_l)d\zb^J \|_{2m}&\lesssim  \left(\|A^q_{(k),s}\|_{\omega,\phi_k}+N_{m,k}+N_{2m,k}\right)\|(\p^{\beta} f_l)d\zb^J \|_{-2m}\\&=O(\|A^q_{(k),s}\|_{\omega,\phi_k}+N_{m,k}+N_{2m,k}).
\end{align*}
\end{proof}
We apply the result to the spectral kernel by setting $A^q_k:= P^q_{k,(C_k)^{-d}}$, which has functional calculus $a^q_k=\mathbbm{1}_{[0,(C_k)^{-d}]}$. The operator norm $\|P^q_{k,(C_k)^{-d}}\|_{\omega,\phi_k,M}$ is clearly less than or equal to $1$ and $N_{m,k}\leq(C_k)^{-d-1}$. For Bergman kernels, set $A^q_k:=P^q_k$ with function calculus $a^q_k=\mathbbm{1}_{\{0\}}$. Note that $\|P^q_k\|_{\omega,\phi_k,M}\leq 1$ and $N_{m,k}=0$ for all $m$ and $k$. 
\begin{defin}[Notations of spectral and Bergman kernels]Fix $d\in\N$. We let $A^q_k= P^q_{k,(C_k)^{-d}}$ and take
$P^{q,s}_{k,(C_k)^{-d}}(z,w):=A^{q,s}_k(z,w)$ which is given by \[P^{q}_{k,(C_k)^{-d}}=P^{q,s}_{k,(C_k)^{-d}}s_k(z)\otimes (s_k(w))^*.\]  Define the \textbf{scaled spectral kernel} by \[
P^{q,s}_{(k),(C_k)^{-d}}(z,w):=(C_k)^{-n}P^{q,s}_{k,(C_k)^{-d}}(\frac{z}{\sqrt{C_k}},\frac{w}{\sqrt{C_k}}).
\]
 Also, we denote by \[
P^q_{k,(C_k)^{-d},s}:=A^q_{k,s}\quad;\quad P^q_{(k),(C_k)^{-d},s}:=A^q_{(k),s}
\]the \textbf{localized spectral projection} and the \textbf{scaled localized spectral projection}, respectively. Furthermore, we define the \textbf{localized spectral kernels} and the \textbf{scaled localized spectral kernels} by \[ P^q_{k,(C_k)^{-d},s}(z,w):= A^q_{k,s}(z,w)\quad;\quad
P^q_{(k),(C_k)^{-d},s}(z,w):= A^q_{(k),s}(z,w),
\]respectively. On the other hand, we set $A^q_k=P^q_{k}$ and denote $P^{q,s}_{k}(z,w):=A^{q,s}_k(z,w)$. Define the \textbf{scaled Bergman kernel} as \[
P^{q,s}_{(k)}(z,w):=(C_k)^{-n}P^{q,s}_{k}(\frac{z}{\sqrt{C_k}},\frac{w}{\sqrt{C_k}}).
\] Also, we define the \textbf{localized Bergman projection} $P^q_{k,s}:=A^q_{k,s}$ and the \textbf{scaled localized Bergman projection} $P^q_{(k),s}:=A^q_{(k),s}$. Denote the \textbf{localized Bergman kernel} and the \textbf{scaled localized Bergman kernel} by\[
P^q_{k,s}(z,w):= A^q_{k,s}(z,w)\quad;\quad
P^q_{(k),s}(z,w):= A^q_{(k),s}(z,w),
\]respectively.\end{defin} By the identity (\ref{0407}), we have the relations 
\begin{align}
  P^{q,s}_{(k),(C_k)^{-d}}(z,w)&=e^{\phi_{(k)}(z)}P^{q}_{(k),(C_k)^{-d},s}(z,w)e^{-\phi_{(k)}(w)} ; \label{ker1}\\
  P^{q,s}_{(k)}(z,w)&=e^{\phi_{(k)}(z)}P^{q}_{(k),s}(z,w)e^{-\phi_{(k)}(w)}.  \label{ker2}
\end{align}\label{ker}
 By Theorem \ref{thm uniform bound}, we have the following corollaries:
\begin{cor}[The local uniform bounds for Bergman and spectral kernels]\label{cor666}
In the localization process introduced before, the scaled spectral kernels \[P^{q}_{(k),(C_k)^{-d},s}(z,w)=C_k^{-n}P^{q}_{k,(C_k)^{-d},s}(z/\sqrt{C_k},w/\sqrt{C_k})\] are locally uniformly bounded in the $\mathscr{C}^{\infty}$-topology on $\Cn\times\Cn$. The result also holds for the scaled localized Bergman kernel $P^q_{(k),s}(z,w)$.
\end{cor}

Moreover, since the constant in Theorem \ref{thm uniform bound} continuously depends on $p$, we can insert $(z,w)=(0,0)=(p,p)$, $\ell=0$ and fix $U=B(1)$ in Theorem \ref{thm uniform bound} to get the following result:
\begin{cor}[local uniform bounds on the diagonal]\label{0410cor}Fix $d\in\N$. For any compact set $K\subset D$, there exists a constant $\Tilde{C}_{K}$ independent of $k$ such that
\begin{equation}
 \sup_{p\in K}\aaa{(C_k)^{-n}P^{q,s}_{k,(C_k)^{-d}}(p,p)}\leq \Tilde{C}_{K}.
\end{equation} Also, the result holds for Bergman kernel case. 
\end{cor}

%% file: Chapter3.tex
\section{Asymptotics of Bergman and spectral kernels}\label{Chapter 3}
We are going to prove the main theorems (cf. Theorem \ref{main theorem}, Theorem \ref{main theorem 2}) from Section \ref{section 4.3} to Section \ref{1233333344}. We will repeat the process in \cite[Chapter 4]{me} with context-specific modifications. In Section \ref{section 3.4}, we will present an idea to show the Bergman kernel asymptotic under a stronger spectral gap condition (cf. Assumption \ref{spectral gap 2}). Before embarking on the proof of main theorems, we need to investigate the extended Laplacian on $\Cn$ and establish the spectral gap.

\subsection{Spectral gaps of the extended Laplacians on $\Cn$}\label{section 4.3}
In this section, we will extend the scaled localized Laplacian $\Box^{q}_{(k),s}$ which is defined on $B(\sqrt{C_k})$ to the whole $\Cn$. The extended localized Laplacian is identical to $\Box^{q}_{(k),s}$ in $B((C_k)^{\epsilon})$ where $\epsilon$ will be determined later in Section \ref{1233333344}. First, by (\ref{loc eq 3}), we note that
\begin{align}\label{ob}
\aaa{\phi_{(k)}-\phi_0}_{\mathscr{C}^2(B(\sqrt{C_k}))} \leq C\dfrac{|z|^3+1}{\sqrt{C_k}}\quad; \quad \ 
\aaa{\omega_{(k)}-\omega_0}_{\mathscr{C}^2(B(\sqrt{C_k}))} \leq C'\frac{|z|+1}{\sqrt{C_k}},
\end{align}where $|\cdot|_{\mathscr{C}^2(B(\sqrt{C_k}))}$ is the usual $\mathscr{C}^{2}$-norm with domain $B(\sqrt{C_k})$.

From now on, we fix a cut-off function denoted by $\chi \in \mathscr{C}^{\infty}_c(\Cn)$ such that its support is contained within the ball $B(2)$, and is identical to $1$ on the ball $B(1)$. Let us choose a number $\epsilon$ such that $0<\epsilon <1/6$ and define the extended metric data on $\Cn$ by
\begin{equation*}
    \tphi_{(k)}(z):=\chi(\dfrac{z}{(C_k)^{\epsilon}})\phi_{(k)}(z)+\left(1-\chi(\dfrac{z}{(C_k)^{\epsilon}})\right)\phi_0(z)
\end{equation*} and the extended Hermitian form by 
\begin{equation*}
    \tomega_{(k)}(z):=\chi(\dfrac{z}{(C_k)^{\epsilon}})\omega_{(k)}(z)+\left(1-\chi(\dfrac{z}{(C_k)^{\epsilon}})\right)\omega_0(z). 
\end{equation*}By (\ref{ob}) and $\epsilon<1/6$, we have the uniform convergences
\begin{equation}\label{1508}
    \aaa{\tphi_{(k)}-\phi_0}_{\mathscr{C}^2(\Cn)} \rightarrow 0 \quad \text{\rm and } \quad \aaa{\tomega_{(k)}-\omega_0}_{\mathscr{C}^2(\Cn)} \rightarrow 0.
\end{equation}Denote 
\begin{align*}
\tpb^{q}_{(k),s}:\Omega^{0,q}(\Cn) \To \Omega^{0,q+1}(\Cn)\quad;\quad
\tpb^{q,*}_{(k),s}:\Omega^{0,q}(\Cn) \To \Omega^{0,q-1}(\Cn)
\end{align*} to be the extended localized Cauchy-Riemann operator and its formal adjoint, given by
\begin{align*}
\tpb^{q}_{(k),s}=\pb^{q}+(\pb \tphi_{(k)})\wedge\cdot\quad;\quad
  \tpb^{q,*}_{(k),s}=\pb^{q,*}_{\tomega_{(k)}}+\left((\pb \tphi_{(k)})\wedge\right)^*_{\tomega_{(k)}},
\end{align*}respectively. Here, $\pb^{q,*}_{\tomega_{(k)}}$ is the formal adjoint of $\pb^q$ with respect to $(\cdot|\cdot)_{\Tilde{\omega}_{(k)}}$. Denote
\[\Box^{q\sim}_{(k),s}=\tpb_{(k),s}^*\tpb_{(k),s}+\tpb_{(k),s}\tpb_{(k),s}^*:\Dom{\Box^{q\sim}_{(k),s}}\subset L^2_{\omega_0}(\Cn,\T) \rightarrow L^2_{\omega_0}(\Cn,\T)\] as the Gaffney extension of the localized Kodaira Laplacian with respect to the Hermitian form $\tomega_{(k)}$ and the weight function $\tphi_{(k)}$. It follows immediately from the constructions that $\pb^q_{(k),s}\equiv\tpb^q_{(k),s}$, 
 $\pb^{q,*}_{(k),s}\equiv\tpb^{q,*}_{(k),s}$ and $\Box^{q\sim}_{(k),s}\equiv\Box^{q}_{(k),s}$ in $B((C_k)^{\epsilon})$. Reasonably, we call the $\Box^{q\sim}_{(k),s}$ \textbf{extended Laplacian}.\\  \par 
Suppose $\lambda_i <0$ for all $i=1,\cdots ,q_0$ ; $\lambda_i>0$ for all $i=q_0+1,\cdots ,n$. Then there exists a constant $c>0$ such that for all $z\in\Cn$, 
\begin{equation}\label{4.4 phi est}
    \dfrac{\p^2 \tphi_{(k)}}{\p z^i \p \zb^i}(z)<-c\quad \forall \, i=1,\cdots ,q_0
\quad \text{\rm and} \quad 
     \dfrac{\p^2 \tphi_{(k)}}{\p z^i \p \zb^i}(z)>c\quad \forall \, i=q_0+1,\cdots ,n.
\end{equation}The following results tell us these estimates create a uniform lower bound of the first eigenvalue of $\Box^{q\sim}_{(k),s}$. 
\begin{lem}\label{4.4 lem 1}
    For $q \neq q_0$, there is a constant $c>0$ such that for all $u \in \Dom \Box^{q\sim}_{(k),s}$,
    \begin{equation*}\left(\Box^{q\sim}_{(k),s}u\mid u\right)_{\tomega_{(k)}}=\|\tpb_{(k),s}u\|^2_{\tomega_{(k)}}+ \|\tpb^{*}_{(k),s}u\|^2_{\tomega_{(k)}}\geq c\|u\|^2_{\tomega_{(k)}}.
    \end{equation*}Therefore, $\|\Box^{q\sim}_{(k),s}u\|_{\tomega_{(k)}}\geq c\|u\|_{\tomega_{(k)}}$.
\end{lem}
\begin{proof}
    Note that for $u\in\Omega^{0,q}_c(\Cn)$, we use (\ref{1508}) to get 
\begin{align}
 \label{4.4 lem 1 (1)}    
\|\tpb_{(k),s}u\|^2_{\tomega_{(k)}}&=\|\left(\pb+(\pb \tphi_{(k)})\wedge\right)u\|^2_{\tomega_{(k)}} \gtrsim \|\left(\pb+(\pb \tphi_{(k)})\wedge\right)u\|^2_{\omega_{0}};\\
\|\tpb^*_{(k),s}u\|^2_{\tomega_{(k)}}&=\|\left(\pb^*_{\tomega_{(k)}}+(\pb \tphi_{(k)})\wedge^*_{\tomega_{(k)}}\right)u\|^2_{\tomega_{(k)}} \gtrsim \|\left(\pb^*_{\omega_{0}}+(\pb \tphi_{(k)})\wedge^*_{\omega_0}\right)u\|^2_{\omega_{0}}.\notag
\end{align} \par 
   Let $u =fd\zb^I$ for some $f \in \mathscr{C}^{\infty}_c(\Cn)$ and $I\in (\N_{0})^q$ be a  strictly increasing multi-index. Since $q \neq q_0$, there exists $i \in \{1,\cdots , n\}$ such that at least one of the following two cases holds:
    \begin{itemize}
        \item $i \notin I$ and  $\lambda_i<0$;
        \item $i \in I$ and  $\lambda_i>0$.
    \end{itemize}If the first case holds, \begin{align}\label{4.4 lem 1 (2)}
 \| \left(\pb+(\pb \tphi_{(k)})\wedge\right)u\|^2_{\omega_0} & \geq \int_{\Cn} \mid \dfrac{\p f}{\p \zb^i}+\dfrac{\p \tphi_{(k)}}{\p \zb^i}f \mid ^2 dm  = \int_{\Cn} (\dfrac{\p f}{\p \zb^i}+\dfrac{\p \tphi_{(k)}}{\p \zb^i}f)(\dfrac{\p \Bar{f}}{\p z^i}+\dfrac{\p \Bar{\tphi}_{(k)}}{\p z^i}\Bar{f}) dm\notag  \\
 &=\int_{\Cn}|\dfrac{\p f}{\p \zb^i}|^2+\Bar{f}\dfrac{\p f}{\p \zb^i}\dfrac{\p \Bar{\tphi}_{(k)}}{\p z^i}+f\dfrac{\p \Bar{f}}{\p z^i}\dfrac{\p \tphi_{(k)}}{\p \zb^i}+|\dfrac{\p \tphi_{(k)}}{\p \zb^i}|^2|f|^2dm. 
\end{align}By integration by part, we get the equations $\int_{\Cn}|\dfrac{\p f}{\p \zb^i}|^2dm=\int_{\Cn}|\dfrac{\p f}{\p z^i}|^2dm$ and
\begin{align*}
    \int_{\Cn}\Bar{f}\dfrac{\p f}{\p \zb^i}\dfrac{\p \Bar{\tphi}_{(k)}}{\p z^i}+f\dfrac{\p \Bar{f}}{\p z^i}\dfrac{\p \tphi_{(k)}}{\p \zb^i}dm = \int_{\Cn}-2|f|^2\dfrac{\p^2 \tphi_{(k)}}{\p z^i \p \zb^i}-f\dfrac{\p \Bar{f}}{\p \zb^i}\dfrac{\p\Bar{\tphi}_{(k)}}{\p z^i}-\Bar{f}\dfrac{\p f}{\p z^i}\dfrac{{\p\tphi}_{(k)}}{\p \zb^i}dm.
\end{align*}Applying these two equations and  
$|\dfrac{\p f}{\p z^i}|^2+|\dfrac{\p \tphi_{(k)}}{\p \zb^i}|^2|f|^2-2|f||\dfrac{\p f}{\p z^i}||\dfrac{{\p\tphi}_{(k)}}{\p \zb^i}| \geq 0$ into (\ref{4.4 lem 1 (2)}),
 \begin{equation}\label{4.4 lem 1 (3)}
  \| \left(\pb+(\pb \tphi_{(k)})\wedge\right)u\|^2_{\omega_0} \geq -2\int_{\Cn}|f|^2\dfrac{\p^2 \tphi_{(k)}}{\p z^i \p \zb^i}dm  \gtrsim -\inf\left(\dfrac{\p^2 \tphi_{(k)}}{\p z^i \p \zb^i}\right)\|f\|^2_{\tomega_{(k)}} .  
 \end{equation}On the other hand, if the second case holds,    
\begin{align}
    \|\left(\pb^*_{\omega_0}+(\pb \tphi_{(k)})\wedge^*_{\omega_0}\right)u\|^2_{\omega_0} &\geq \int_{\Cn} (-\dfrac{\p f}{\p z^i}+\dfrac{\p \tphi_{(k)}}{\p z^i}f)(-\dfrac{\p \Bar{f}}{\p \zb^i}+\dfrac{\p \Bar{\tphi}_{(k)}}{\p \zb^i}\Bar{f}) dm \notag\\
 &=\int_{\Cn}|\dfrac{\p f}{\p z^i}|^2-\Bar{f}\dfrac{\p f}{\p z^i}\dfrac{\p \Bar{\tphi}_{(k)}}{\p \zb^i}-f\dfrac{\p \Bar{f}}{\p \zb^i}\dfrac{\p \tphi_{(k)}}{\p z^i}+|\dfrac{\p \tphi_{(k)}}{\p z^i}|^2|f|^2dm. \label{123444}
\end{align}By integration by part again, we have  $\int_{\Cn}|\dfrac{\p f}{\p z^i}|^2dm=\int_{\Cn}|\dfrac{\p f}{\p \zb^i}|^2dm$ and 
\begin{equation*}
 \int_{\Cn}-\Bar{f}\dfrac{\p f}{\p z^i}\dfrac{\p \Bar{\tphi}_{(k)}}{\p \zb^i}-f\dfrac{\p \Bar{f}}{\p \zb^i}\dfrac{\p \tphi_{(k)}}{\p z^i}dm = \int_{\Cn}2|f|^2\dfrac{\p^2 \tphi_{(k)}}{\p z^i \p \zb^i}+f\dfrac{\p \Bar{f}}{\p z^i}\dfrac{\p\Bar{\tphi}_{(k)}}{\p \zb^i}+\Bar{f}\dfrac{\p f}{\p \zb^i}\dfrac{{\p\tphi}_{(k)}}{\p z^i}dm.   
\end{equation*}Combining equations above and  $|\dfrac{\p f}{\p \zb^i}|^2+|\dfrac{\p \tphi_{(k)}}{\p z^i}|^2|f|^2-2|f||\dfrac{\p f}{\p \zb^i}||\dfrac{{\p\tphi}_{(k)}}{\p z^i}| \geq 0$, 
\begin{equation}\label{4.4 lem 1 (4)}
 \|\left(\pb^*_{\omega_0}+(\pb \tphi_{(k)})\wedge^*_{\omega_0}\right)u\|^2_{\omega_0} \geq 
 2 \int_{\Cn}|f|^2 \dfrac{\p^2 \tphi_{(k)}}{\p z^i \p \zb^i} dm \gtrsim \inf\left(\dfrac{\p^2 \tphi_{(k)}}{\p z^i \p \zb^i}\right)\|f\|^2_{\tomega_{(k)}} .  
\end{equation}By (\ref{4.4 lem 1 (1)}),(\ref{4.4 lem 1 (3)}) and (\ref{4.4 lem 1 (4)}), we have completed the proof for the case $u\in \Omega^{0,q}_c(\Cn)$. Next, we can prove the lemma by density argument. The density argument here is somehow technical and based on the \textbf{Friedrich's Lemma} (cf.\cite[Chapter 7, Lemma 3.3]{02}). For the details of approximation, readers may consult \cite[Lemma 5]{Hou}. 
\end{proof}

\begin{cor}\label{4.4 cor N}
 For $q \neq q_0$, the extended Laplacians $\Box^{q\sim}_{(k),s}$ is bijective and has  inverses \[N^{q}_{k}:L^2_{\tomega_{(k)}}(\Cn,\T) \rightarrow \Dom \Box^{q\sim}_{(k),s}\] which is a $k$-uniformly bounded operator.    
\end{cor}
\begin{proof}
    According to Lemma \ref{4.4 lem 1}, $\Box^{q\sim}_{(k),s}$ is injective. To show the surjectivity, we choose an arbitrary $v\in L^2_{\tomega_{(k)}}(\Cn,\T)$ and consider the linear functional $\Tl_{v}$ on $\Rang \, \Box^{q\sim}_{(k),s}$ given by 
    \[
    \Tl_{v}(\Box^{q\sim}_{(k),s}\,u)=\left(u \mid v\right)_{\tomega_{(k)}} \quad \forall u\in \Dom \, \Box^{q\sim}_{(k),s}.
    \]Lemma \ref{4.4 lem 1} implies that $\|\Tl_{v}\|_{\tomega_{(k)}} \leq \dfrac{\|v\|_{\tomega_{(k)}}}{c}$ for a constant $c$ independent of $v$ and $k$. By the Hahn-Banach Theorem, the functional $\Tl_{v}$ can be extended to a bounded linear functional on $L^2_{\tomega_{(k)}}(\Cn,\T)$ with the same norm. By Riesz representation theorem, there exists a representative $\Tilde{v} \in L^2_{\tomega_{(k)}}(\Cn,\T)$ such that 
    \[
    \left(u \mid v \right)_{\tomega_{(k)}}=\Tl_{v}(\Box^{q\sim}_{(k),s}u)=\left(\Box^{q\sim}_{(k),s}u\mid\Tilde{v}\right)_{\tomega_{(k)}} \quad \forall u\in \Dom \, \Box^{q\sim}_{(k),s}.
    \]This means $\Box^{q\sim}_{(k),s}\Tilde{v}=v$ which proves the surjectivity. Define $N^{q}_{k}$ such that $N^{q}_{k}v=\Tilde{v}$.  Lemma \ref{4.4 lem 1} implies $\|N^{q}_{k}\|_{\tomega_{(k)}} \leq C$ for a constant $C$ independent of $k$.
\end{proof}
We have shown that when $q\neq q_0$, the extended Laplacian $\Box^{q\sim}_{(k),s}$ has a uniform spectral gap $\text{ spec }\Box^{q\sim}_{(k),s}\subset [c,\infty)$ for a positive constant $c$ independent of $k$. Next, in the case $q=q_0$, we should prove that the uniform spectral gap also holds in the sense that $\text{\rm spec }\Box^{q\sim}_{(k),s}\subset \{0\}\cup [c,\infty)$. Define \[
\Tilde{B}^{q}_{(k),s}:L^2_{\tomega_{(k)}}(\Cn,\T) \rightarrow \Ker \, \Box^{q\sim}_{(k),s} \subset L^2_{\tomega_{(k)}}(\Cn,\T)
\]to be the Bergman projection. The following representation of $\Tilde{B}^{q}_{(k),s}$ is standard.
\begin{thm}\label{4.4 thm B=I-}(Hodge decomposition) We have the expression
\begin{equation}\label{4.4 thm B (1)}
\Tilde{B}^{q_0}_{(k),s}=\Id-\tpb^{q_0-1}_{(k),s}N^{q_0-1}_{k}\tpb^{q_0,*}_{(k),s}-\tpb^{q_0+1,*}_{(k),s}N^{q_0+1}_{k}\tpb^{q_0}_{(k),s} \quad \textit{on}\,\, \Omega^{0,q}_c(\Cn).    
\end{equation}
Here, $N^{q}_{k}$ is the inverse of the Laplacian $\Box^{q\sim}_{k,s}$ established in Corollary \ref{4.4 cor N}.
\end{thm}
\begin{proof}Note that
    \begin{align*}
&\Box^{q_0\sim}_{(k),s}\left(\Id-\tpb^{q_0-1}_{(k),s}N^{q_0-1}_{k}\tpb^{q_0,*}_{(k),s}-\tpb^{q_0+1,*}_{(k),s}N^{q_0+1}_{k}\tpb^{q_0}_{(k),s}\right)\\&=\tpb_{(k),s}\tpb^*_{(k),s}+\tpb_{(k),s}^*\tpb_{(k),s}-\tpb_{(k),s}\tpb^*_{(k),s}\tpb_{(k),s} N^{q_0-1}_{k}\tpb^*_{(k),s}-\tpb^{*}_{(k),s}\tpb_{(k),s}\tpb_{(k),s}^{*}N^{q_0+1}\tpb_{(k),s} \\
       &=\tpb_{(k),s}\tpb^*_{(k),s}+\tpb^*_{(k),s}\tpb_{(k),s}-\tpb_{(k),s} \Box^{q_0-1\sim}_{(k),s} N^{q_0-1}_k\tpb^*_{(k),s}-\tpb^{*}_{(k),s}\Box^{q_0+1\sim}_{(k),s}N^{q_0+1}_k\tpb_{(k),s}\\&=\tpb_{(k),s}\tpb^*_{(k),s}+\tpb^*_{(k),s}\tpb_{(k),s}-(\tpb_{(k),s}\tpb^*_{(k),s}+\tpb^*_{(k),s}\tpb_{(k),s})=0.
    \end{align*}So the right-hand side of (\ref{4.4 thm B (1)}) has its image in $\Ker \Box^{q\sim}_{(k),s}$. It remains to show that $\Rang \left(\tpb^{q_0-1}_{(k),s}N^{q_0-1}_k\tpb^{q_0,*}_{(k),s}-\tpb^{q_0+1,*}_{(k),s}N^{q_0+1}_k\tpb^{q_0}_{(k),s}\right) \perp \Ker \Box^{q_0\sim}_{(k),s}$. Now, given $u\in \Omega^{0,q}_c(\Cn)$ and $v\in \Ker \Box^{q_0\sim}_{(k),s}$, since $\tpb^*_{(k),s}v=\tpb_{(k),s}v=0$,
    \begin{multline*}
       \left((\tpb_{(k),s} N^{q_0-1}_k\tpb^*_{(k),s}-\pb^{*}N^{q_0+1}_k\pb)u \mid v\right)_{\tomega_{(k)}}\\ =\left(N^{q_0-1}_k\tpb_{(k),s} u \mid \tpb^*_{(k),s} v\right)_{\tomega_{(k)}}+\left(N^{q_0+1}_k\tpb_{(k),s} u \mid \tpb_{(k),s}v\right)_{\tomega_{(k)}}=0. 
    \end{multline*}
\end{proof}
We now deduce some identities which will be frequently utilized. Compute that 
\begin{align*}
\|\tpb_{(k),s}\tpb^*_{(k),s}N^{q_0-1}_{k}\tpb^{*}_{(k),s}u\|^2_{\tomega_{(k)}}&= \left(\tpb^*_{(k),s}\tpb_{(k),s}\tpb^*_{(k),s}N^{q_0-1}_{k}\tpb^{*}_{(k),s}u \mid \tpb^*_{(k),s}N^{q_0-1}_{k,s}\tpb^{*}_{(k),s}u \right)_{\tomega_{(k)}}\\
     &=\left(\tpb^*_{(k),s}\Box^{q_0-1\sim}_{(k),s}N^{q_0-1}_{k}\tpb^{*}_{(k),s}u \mid \tpb^*_{(k),s}N^{q_0-1}_{k}\tpb^{*}_{(k),s}u \right)_{\tomega_{(k)}}\\&=\left(\tpb^*_{(k),s}\tpb^{*}_{(k),s}u \mid \tpb^*_{(k),s}N^{q_0-1}_{k}\tpb^{*}_{(k),s}u \right)_{\tomega_{(k)}}=0,    
\end{align*}for all $u\in\Omega^{0,q_0}_c(\Cn)$. Similarly, we can compute that $\|\tpb^{*}_{(k),s}\tpb_{(k),s}N^{q_0+1}_{k}\tpb_{(k),s}u\|^2_{\tomega_{(k)}}=0$ for all $u\in\Omega^{0,q_0}_c(\Cn)$. Hence, we have 
\begin{equation}\label{4.4 eq 1}
\tpb_{(k),s}\tpb^{*}_{(k),s}N^{q_0-1}_{k}\tpb^{*}_{(k),s}=0\quad;\quad \tpb^{*}_{(k),s}\tpb_{(k),s}N^{q_0+1}_{k}\tpb_{(k),s}=0 \quad \text{on }\, \Omega^{0,q_0}_c(\Cn).  
\end{equation}
Moreover, we can apply the two equations above to see that 
\begin{equation}\label{4.4 eq 2}
\tpb^*_{(k),s}\tpb_{(k),s}N^{q_0-1}_{k}\tpb^*_{(k),s}=\tpb^*_{(k),s}\quad ;\quad \tpb_{(k),s}\tpb^*_{(k),s}N^{q_0-1}_{k}\tpb_{(k),s}=\tpb_{(k),s}\quad\text{on }\, \Omega^{0,q_0}_c(\Cn).\end{equation}

\begin{thm}[uniform spectral gap]\label{4.4 thm spectral gap} 
There exists a constant $c$ independent of $k$ such that\[\|\Tilde{B}^{q_0}_{(k),s} u-u\|^2_{\tomega_{(k)}} \leq c\left(\|\tpb^{q_0,*}_{(k),s}u\|^2_{\tomega_{(k)}}+\|\tpb^{q_0}_{(k),s}u\|^2_{\tomega_{(k)}}\right)  \quad \text{\rm on }\, \Omega^{0,q_0}_c(\Cn).\]
\end{thm}
\begin{proof}
  By Lemma \ref{4.4 thm B=I-}, \[\Tilde{B}^{q_0}_{(k),s}-I=-\tpb^{q_0-1}_{(k),s}N^{q_0-1}_{k}\tpb^{q_0,*}_{(k),s}-\tpb^{q_0+1,*}_{(k),s}N^{q_0+1}_{k}\tpb^{q_0}_{(k),s} \quad \text{\rm on } \, \Omega^{0,q_0}_c(\Cn).\]Given $u \in \Omega^{0,q_0}_c(\Cn)$, 
  \begin{align*}
   \|\tpb^{q_0-1}_{(k),s}N^{q_0-1}_{k}\tpb^{q_0,*}_{(k),s}u\|^2_{\tomega_{(k)}} &= \left(N^{q_0-1}_{k}\tpb^{*}_{(k),s}u\mid \tpb^{*}_{(k),s} \tpb_{(k),s}N^{q_0-1}_{k}\tpb^{*}_{(k),s}u\right)_{\tomega_{(k)}}\\&\leq \|N^{q_0-1}_{k}\tpb^{*}_{(k),s}u\|_{\tomega_{(k)}} \|\tpb^{*}_{(k),s}\tpb_{(k),s}N^{q_0-1}_{k}\tpb^{*}_{(k),s}u\|_{\tomega_{(k)}}\\&\lesssim \|\tpb^{*}_{(k),s}u\|^2_{\tomega_{(k)}}.
  \end{align*}The last inequality is from Corollary \ref{4.4 cor N} and (\ref{4.4 eq 2}). Symmetrically, we can show that
  \begin{equation*}
\|\tpb^{q_0+1,*}_{(k),s}N^{q_0+1}_k\tpb^{q_0}_{(k),s}u\|^2_{\tomega_{(k)}}\leq c\|\tpb^{q_0}_{(k),s}u\|^2_{\tomega_{(k)}}.
  \end{equation*}The two estimates above imply the theorem.
\end{proof}

\subsection{Proof of main theorems }\label{ssec}
Since all the arguments in this section are valid for both the scaled localized Bergman kernel $P^q_{(k),s}(z,w)$ and the scaled localized spectral kernel $P^q_{(k),(C_k)^{-d},s}(z,w)$. For simplicity, we represent both the scaled localized spectral and Bergman kernels as $P^q_{(k),s}(z,w)$ in this section.

Recall Corollary \ref{cor666}. By the Arzela-Ascoli Theorem, we know that every subsequence of $P^q_{(k),s}(z,w)$  has a convergent subsequence in the $\mathscr{C}^{\infty}$-topology. To show that $P^q_{(k),s}(z,w)$ is itself a uniformly convergent sequence, it suffices to show that every convergent subsequence of $P^q_{(k),s}(z,w)$ converges to the same limit. 

To prove the main theorems, we assume that $P^q_{(k),s}(z,w)$ converges locally uniformly to $P^q_{s}(z,w)$ in the $\mathscr{C}^{\infty}$-topology. Although we do not yet know the kernel section $P^q_{s}(z,w)$, we will demonstrate that it must be the Bergman kernel $P^q_{0,s}(z,w)$ in the model case on $\Cn$, which will be introduced later. If that is the case, then we have proved the main theorems by Theorem \ref{4.1 Model theorem}, and identities (\ref{ker1}), (\ref{ker2}). 

We now formulate the Bergman kernel in the model case. Recall the Laplacian $\Box^q_0$ and the localized Laplacian $\Box^q_{0,s}$ defined in (\ref{lapl mod 1}) and (\ref{lapl mod 2}), respectively. Denote by $P^q_0(z,w)$ the Bergman kernel of $\Box^q_0$ and by $P^q_{0,s}(z,w)$ the localized Bergman kernel of $\Box^q_{0,s}$. Note that \[
P^q_{0,s}(z,w)=e^{-\phi_0(z)}P^q_{0}(z,w)e^{\phi_0(w)}.
\]For $q\in\{1,\cdots,n\}$ and $\alpha\in(\N_0)^{n}$, denote
\[
z^{\alpha}_q:=(\zb^1)^{\alpha_1}\cdots(\zb^q)^{\alpha_q}(z^{q+1})^{\alpha_{q+1}}\cdots (z^{n})^{\alpha_{n}}.
\]
We now introduce a theorem that describes the Bergman kernel in the model case.

\begin{thm}\label{4.1 Model theorem}\cite[Theorem 4.2]{me}
Consider the trivial vector bundle $\T \otimes \Cs \rightarrow \Cn$ endowed with the standard Hermitian form $\omega_0$ and the weight function $\phi_0$. In the case $p\in M(q)$, 
 we assume $\lambda_i<0$ for all $i\leq q$ and $\lambda_i>0$ for all $i>q$. The localized Bergman kernel $P^{q}_{0,s}(z,w)$ of the model case is given by
\begin{equation*}
    \dfrac{|\lambda_1 \cdots \lambda_n|}{\pi^n}\,e^{2(\sum_{i=1}^{q}|\lambda_i|\zb^i w^i+\sum_{i=q+1}^{n}|\lambda_i|z^i\wb^i)-\sum_{i=1}^{n}|\lambda_i|(|z^i|^2+|w^i|^2)}(d\zb^1\wedge \cdots \wedge d\zb^q) \otimes (\dfrac{\p}{\p \wb^{1}} \wedge \cdots \wedge \dfrac{\p}{\p \wb^{q}}).
\end{equation*}Furthermore,
\[
\{\Psi_{\alpha}:=\sqrt{\dfrac{2^{|\alpha|}[\lambda]^{\alpha+1}}{\pi^n \alpha!}}z^{\alpha}_{q}e^{-\sum_{i=1}^{n}|\lambda_i||z^i|^2}d\zb^1\wedge \cdots \wedge d\zb^q\}_{\alpha \in \N_0^n} 
\]is the orthonormal basis of $\Ker \Box^{q}_{0,s}\subset L^2_{\omega_0}(\Cn,\T)$.\par However, if $p\notin M(q)$, then
\[
\Ker \Box^{q}_{0,s}=\{0\}
\,\, \text{ and hence }\, \Bsko(z,w) \equiv 0.\]
\end{thm}

We now begin by translating $P^q_{s}(z,w)$ from an unknown kernel section to an operator that acts on the Hilbert space $L^2_{\omega_0}(\Cn,\T)$. 
\begin{defin}\label{Berg def}Define the integral operator $P^q_s$ as
 \begin{equation*}
        P^q_s u(z) := \int_{\Cn}P^q_s(z,w)u(w)dm(w) \quad\text{\rm for all }\, u\in L^2_{\omega_0}(\Cn,\T).
    \end{equation*}   
\end{defin}Clearly, the integral converges by the assumption $P^q_{(k),s}(z,w)\To P^q_{s}(z,w)$.
\begin{lem}[Well-definition of the integral operator]\label{4.2 Lem B op}
    The integral operator
    \begin{equation*}
        P^q_s:L^2_{\omega_0}(\Cn,T^{*,(0,q)}\Cn) \rightarrow L^2_{\omega_0}(\Cn,T^{*,(0,q)}\Cn)
    \end{equation*}is a bounded linear map with its operator norm smaller than $1$.
\end{lem}
\begin{proof}
      Let $u,v\in\Omega^{0,q}_c(\Cn)$ and $U\subset \Cn$ be a bounded open set containing $\supp u$ and $\supp v$. Note that
    \begin{align*}
        \left(v\mid P^q_s u\right)_{\omega_0}&=\int_{U}\int_{U}\la v(z) | P^q_s(z,w)u(w) \ra_{\omega_0}2^{2n}dm(w)dm(z).
    \end{align*}Let $\varepsilon>0$. By the fact that $P^q_{(k),s}(z,w) \rightarrow P^q_s(z,w)$ local uniformly, the above integral can be dominated as
    \begin{equation*}
    \aaa{\left(v\mid P^q_s u\right)_{\omega_0}}\leq(1+\varepsilon)\aaa{\left(v\mid P^q_{(k),s} u\right)_{\omega_0}}, \end{equation*} 
    for large enough $k$. By inserting $A^q_k=P^q_k$ in (\ref{uni 1}) and applying (\ref{uni 4}) and the fact $\|P^q_k\|_{\omega,\phi_k}\leq 1$, we have
    \[
    \aaa{\left(v\mid P^q_{(k),s} u\right)_{\omega_0}}\leq\|v\|_{\omega_0}\|P^q_{(k),s}\|_{\omega_0,U}\|u\|_{\omega_0}\leq (1+\varepsilon)\|v\|_{\omega_0}\|u\|_{\omega_0},
    \]for large enough $k$.
    Since $\varepsilon>0$ and $v$ are arbitrary, we have $\|P^q_su\|_{\omega_0}\leq \|u\|_{\omega_0}$. The Lemma follows by density argument. 
\end{proof}

\begin{lem}\label{4.2 thm bd op BoxB=0}
 $P^q_s $ is a bounded linear map \[P^q_s:L^{2}_{\omega_0}(\Cn,\T)\To \Ker \Box^{q}_{0,s}.\]
\end{lem}
\begin{proof}
We may assume  $u\in\Omega^{0,q}_c(\Cn)$ by density argument and fix a cut-off function $\chi\in\Omega^{0,q}_c(\Cn)$. By the fact $\Box^q_{(k),s}\To\Box^q_{0,s}$ locally uniformly and the assumption $P^q_{(k),s}(z,w)\To P^q_{s}(z,w)$,
\[
\|\chi\Box^q_{0,s}P^q_s u\|_0\lesssim \|\chi \Box^q_{(k),s}P^q_{(k),s}u\|_{0}=0.
\]
\end{proof}
 
To complete the proof of the main theorems, by the uniqueness of Schwartz kernels, it remains to prove that \[P^q_s:L^2_{\omega_0}(\Cn,\T)\rightarrow\Ker\Box^{q}_{0,s}\] is an orthogonal projection. By Lemma
\ref{4.2 thm bd op BoxB=0}, it is left to prove the following statement (see \cite[theorem 3.1 in section 3.1]{Yos}):
\begin{statement}\label{4.2 statement 2} 
$P^q_s u=u$ for all $u \in \Ker \Box^{q}_{0,s}$.
\end{statement}
\subsection{Proof of Statement \ref{4.2 statement 2}}\label{1233333344}
 In the case $p\notin M(q)$, Theorem \ref{4.1 Model theorem} tells us that $\Ker\Box^{q}_{0,s}=\{0\}$, and therefore $P^q_s$ is a zero map by Lemma \ref{4.2 thm bd op BoxB=0}. Consequently, Statement \ref{4.2 statement 2} automatically holds. We have completed the proof of the main theorem for the vanishing case (cf. Theorem \ref{main theorem}).

 We now focus on proving the case $p\in M(q)$. First, for the spectral projection $P^q_{k,(C_k)^{-d}}$, we have the estimate:\begin{equation}\label{natural}
  \|\left(\Id-P^q_{k,(C_k)^{-d}}\right)u\|_{\omega,\phi_k,M}\leq (C_k)^{d}\left(\Box^q_k u|u\right)_{\omega,\phi_k,M} \quad \text{for all }\, u\in\Omega^{0,q}_c(M,L_k).  
 \end{equation} In this section, we assume the spectral gap condition (cf. Def.\ref{spectral gap}) in the case of Bergman kernel and continue the convention of notation in Section \ref{ssec}. We use the notation $P^q_{k}$ to represent both the Bergman projection $P^q_{k}$ and the spectral projection $P^q_{k,(C_k)^{-d}}$. Similarly, $P^q_{k,s}(z,w)$ represents both the localized Bergman kernel $P^q_{k,s}(z,w)$ and the localized spectral kernel $P^q_{k,(C_k)^{-d},s}(z,w)$. Also, $P^q_{(k),s}(z,w)$ represents both the scaled localized Bergman kernel $P^q_{(k),s}(z,w)$ and the scaled localized spectral kernel $P^q_{(k),(C_k)^{-d},s}(z,w)$. By (\ref{natural}) and spectral gap condition (cf. Def \ref{spectral gap}) for the case of Bergman kernel, we can write
\begin{equation}\label{Hohohoho}
        \|\left(\Id-P^{q}_{k}\right)u\|^2_{\omega,\phi_k,M}\lesssim  ({C_k})^{d}\left( \Box^{q}_{k}u\mid u\right)_{\omega,\phi_k,M} \quad \text{for all }\, u\in\Omega^{0,q}_c(D,L_k).
    \end{equation}The spectral and Bergman kernels share the same estimate (\ref{Hohohoho}). The remaining proof for the main theorem (cf. Theorem \ref{main theorem 2}) is valid for both the Bergman and the spectral kernels. We now embark on the proof. By rearrangement, let $\lambda_i <0$ for all $i=1,\cdots,q$ and $\lambda_i>0$ for all $i=q+1,\cdots,n$. For $\alpha\in(\N_0)^n$, denote \[z^{\alpha}_q:=(\zb^1)^{\alpha_1}\cdots(\zb^q)^{\alpha_q}(z^{q+1})^{\alpha_{q+1}}\cdots (z^{n})^{\alpha_{n}}.\] By theorem \ref{4.1 Model theorem}, to show the Statement \ref{4.2 statement 2} for $p\in M(q)$, we may assume that $u$ is of the form
\[
u=z^{\alpha}_qe^{-\sum_{i=1}^n|\lambda_i||z^i|^2}d\zb^I,
\] where $I:=(1,\cdots,q)$. We adopt the settings in Section \ref{section 4.3}. It is important to note that in the construction of $\tomega_{(k)}$ and $\tphi_{(k)}$, we impose the condition that $0<\epsilon<1/6$. Now, we require 
\[
0<\epsilon<\min\{\dfrac{1}{2n+1},\dfrac{1}{6}\}.
\] The reason is in the proof of Theorem \ref{4.5 thm 1}.\par 
We now establish the notations of cut-off functions. Recall that $\chi \in \mathscr{C}^{\infty}_c(\Cn)$ is the cut-off function fixed at the beginning of Section \ref{section 4.3}. Choose $\rho \in \mathscr{C}^{\infty}_c(\Cn)$ as another cut-off function such that $\supp \rho \subset\{z\in\Cs; 2/7<|z|<1\}$ and $\rho\equiv 1$ on $\{z\in\Cs; 3/7<|z|<6/7\}$. Construct a sequence of cut-off functions by
\begin{equation}
    \chi_k(z):=\chi(\dfrac{7z}{(C_k)^{\epsilon}})\quad;\quad\Tilde{\chi}_k(z):=\chi(\dfrac{7z}{3(C_k)^{\epsilon}})\quad;\quad\rho_k(z):=\rho(\dfrac{z}{(C_k)^{\epsilon}}).
\end{equation}Observe that \[ \supp \chi_k \subset \{z\in\Cs;\,|z|<(2/7)(C_k)^{\epsilon}\} \quad ; \quad \supp \Tilde{\chi}_k \subset \{z\in\Cs;\,|z|<(6/7)(C_k)^{\epsilon}.\}\]  The derivatives of $\Tilde{\chi}_k$ are supported in the annuli $\{z\in\Cs;\,(3/7)(C_k)^{\epsilon}<|z|<(6/7)(C_k)^{\epsilon}\}$ and the support of $\rho_k$ are in the annuli $\{z\in\Cs;\,(2/7)(C_k)^{\epsilon}<|z|<(C_k)^{\epsilon}\}$. 
 Next, we define 
\begin{equation}
u_{(k)}:=\Tilde{\chi}_k \ \Tilde{B}^{q}_{(k),s} \ \chi_k u.
\end{equation}
Our objective is to show the convergence $u_{(k)}\To u$ in $L^2_{\omega_0}(\Cn,\T)$ as $k\To\infty$.
\begin{lem}\label{4.5 thm 1 lem 1}
    \begin{equation}
        \aaa{\tpb^{q}_{(k),s}u(z)}_{\omega_0}+\aaa{\tpb^{q,*}_{(k),s}u(z)}_{\omega_0}=O(\dfrac{1}{\sqrt{C_k}}) \quad \text{for }\, |z|< (C_k)^{\epsilon}.
    \end{equation}
\end{lem}
\begin{proof}
 Denote $u=:fd\zb^I$ where $f=z^{\alpha}_{q}e^{-\sum_{i=1}^n|\lambda_i||z^i|^2}$. Denote $a_1(z)$ and $a_2(z)$ as the absolute maximum of the coefficients of the differential operators $\tpb_{(k),s}-\pb_{0,s}$ and $\tpb^*_{(k),s}-\pb^*_{0,s}$ at a point $z\in\Cn$, respectively. By (\ref{ob}), we can see that 
    \begin{equation*}
        \aaa{a_i(z)}\lesssim \dfrac{|z|^3+1}{\sqrt{C_k}}\quad \forall\, |z|<2(C_k)^{\epsilon} \quad \text{and}\quad \aaa{a_i(z)}=0\quad\forall\,|z|>2(C_k)^{\epsilon}. \end{equation*}Because any derivatives of $u$ decay exponentially as $|z|$ goes to infinity, there is a constant $c > 0$ such that 
     \begin{equation*}
    \aaa{(\tpb_{(k),s}-\tpb_{0,s})u(z)}_{\omega_0} \lesssim \dfrac{|z|^3+1}{\sqrt{C_k}}e^{-c|z|^2} \quad \text{\rm and} \quad \aaa{(\tpb^*_{(k),s}-\tpb^*_{0,s})u(z)}_{\omega_0} \lesssim \dfrac{|z|^3+1}{\sqrt{C_k}}e^{-c|z|^2}, 
     \end{equation*} for all $z\in\Cn$. Since $|z|^3 e^{-c|z|^2}$ is a bounded function, we have completed the proof.
\end{proof}
\begin{lem}\label{4.5 thm 1}$\|u_{(k)}-u\|_{\omega_{0}} \rightarrow 0$ as $k\To\infty$.
\end{lem}
\begin{proof}
     Note that
     $
     \|u_{(k)}-u\|_{\omega_{0}} \lesssim \|u_{(k)}-\tchi_k u\|_{\tomega_{(k)}}+\|\tchi_{k}u-u\|_{\tomega_{(k)}}
     $. The second term tends to zero by dominated convergence theorem. For the first term,
     \begin{align*}
     \|u_{(k)}-\tchi_ku\|_{\tomega_{(k)}}= \|\tchi_k (\Tilde{B}^{q}_{(k),s} \chi_k u-u)\|_{\tomega_{(k)}}&\leq \|\Tilde{B}^{q}_{(k),s} \chi_k u-u\|_{\tomega_{(k)}}\\ &\leq \|\Tilde{B}^{q}_{(k),s} \chi_k u-\chi_k u\|_{\tomega_{(k)}}+\|\chi_k u-u\|_{\tomega_{(k)}}.    \end{align*}
      Since the second term on the right-hand side tends to zero, we only need to estimate $\|\Tilde{B}^{q}_{(k),s}\chi_k u-\chi_k u\|_{\tomega_{(k)}}$. By Theorem \ref{4.4 thm spectral gap},
     \[
     \|\Tilde{B}^{q}_{(k),s} \chi_k u-\chi_k u\|^2_{\tomega_{(k)}} \lesssim \|\tpb^*_{(k),s}\chi_ku\|^2_{\tomega_{(k)}}+\|\tpb_{(k),s}\chi_ku\|^2_{\tomega_{(k)}}.
     \]It remains to claim $\|\tpb^*_{(k),s}\chi_ku\|^2_{\tomega_{(k)}} \To 0$ and $\|\tpb_{(k),s}\chi_ku\|^2_{\tomega_{(k)}} \To 0$. For $\|\tpb_{(k),s}\chi_ku\|^2_{\tomega_{(k)}}$, we compute that $\tpb_{(k),s}\chi_ku=(\pb \chi_k)\wedge u+\chi_k \tpb_{(k),s}u$ and then 
\begin{align*}\|\tpb_{(k),s}\chi_ku\|^2_{\tomega_{(k)}} &\leq\int_{\{|z|<C_k^{\epsilon}/7\}}|\tpb_{k,s}u|_{\tomega_{(k)}}^2dV_{\tomega_{(k)}}+\int_{\{C_k^{\epsilon}/7<|z|<2C_k^{\epsilon}/7\}}|(\pb \chi_k)\wedge u+\chi_k \tpb_{k,s}u|^2_{\tomega_{(k)}}dV_{\tomega_{(k)}} \\
     &\lesssim
    \int_{\{|z|<2C_k^{\epsilon}/7\}}|\tpb_{(k),s}u|_{\omega_{0}}^2dm+\int_{\{C_k^{\epsilon}/7<|z|<2C_k^{\epsilon}/7\}}|u|^2_{\omega_{0}}dm.
     \end{align*}Clearly, the second term $\int_{\{C_k^{\epsilon}/7<|z|<2C_k^{\epsilon}/7\}}|u|^2_{\omega_{0}}dm$ tends to zero. By Lemma \ref{4.5 thm 1 lem 1}  and the setting $\epsilon<1/(2n)$, the first term can be dominated by
   \[
   \int_{\{|z|<2C_k^{\epsilon}/7\}}\aaa{\tpb_{(k),s}u}_{\omega_{0}}^2dm\lesssim \dfrac{(C_k)^{2n\epsilon}}{C_k}\To 0.
   \] We have proven $\|\tpb_{(k),s}\chi_ku\|^2_{\tomega_{(k)}} \To 0$. 
   We can show $\|\tpb^*_{(k),s}\chi_ku\|_{\tomega_{(k)}}\To 0$ in same way. 
\end{proof}
In the next step, we will display $P^q_{(k),s} u_{(k)} -u_{(k)}\To 0$ in $L^2_{\tomega_{(k)}}(\Cn,\T)$. First, we need to verify the following Lemma:

\begin{lem}\label{4.5 thm2 lem1}
Consider the functional
$\rho_k \Tilde{B}^{q}_{(k),s} \chi_k: L^2_{\tomega_{(k)}}(\Cn,\T) \rightarrow L^2_{\tomega_{(k)}}(\Cn,\T)$. For any $d\in \N$, the operator norms have the asymptotic:
\[
\|\rho_k \Tilde{B}^{q}_{(k),s} \chi_k\|_{\tomega_{(k)}}=O((C_k)^{-d})
\]
\end{lem}
\begin{proof}
For any $u\in \Omega^{0,q}_c(\Cn)$, by Theorem \ref{4.4 thm B=I-},
\begin{align}\label{4.5 lem1 (0)}
\rho_k\Tilde{B}^{q}_{(k),s}\chi_ku&=\rho_k\left(\Id-\tpb^*_{(k),s}N^{q+1}_k\tpb_{(k),s}-\tpb_{(k),s}N^{q-1}_k\tpb^*_{(k),s}\right)\chi_ku \\ &= -\rho_k\tpb^*_{(k),s}N^{q+1}_k\tpb_{(k),s}\chi_ku-\rho_k\tpb_{(k),s}N^{q+1}_k\tpb^*_{(k),s}\chi_k u.\notag 
 \end{align}Now, we first aim to estimate $\|\rho_k\tpb^{*}_{(k),s}N^{q+1}_k\tpb_{(k),s}\chi_ku\|_{\tomega_{(k)}}$. Observe that
 \begin{align*}
\|\rho_k\tpb^{*}_{(k),s}N^{q+1}_k\tpb_{(k),s}\chi_ku\|^2_{\tomega_{(k)}}&=\left(\rho_k\tpb^{*}_{(k),s}N^{q+1}_k\tpb_{(k),s}\chi_ku\mid \rho_k\tpb^{*}_{(k),s}N^{q+1}_k\tpb_{(k),s}\chi_ku\right)_{\tomega_{(k)}}\\&
=\left(N^{q+1}_k\tpb_{(k),s}\chi_ku\mid \tpb_{(k),s}\rho^2_k\tpb^{*}_{(k),s}N^{q+1}_k\tpb_{(k),s}\chi_ku\right)_{\tomega_{(k)}}\notag\\&=\left(\Tilde{\rho}_kN^{q+1}_k\tpb_{(k),s}\chi_ku\mid \tpb_{(k),s}\rho^2_k\tpb^{*}_{(k),s}N^{q+1}_k\tpb_{(k),s}\chi_ku\right)_{\tomega_{(k)}}\notag\\&\leq\|\Tilde{\rho}_k N^{q+1}_k\tpb_{(k),s}\chi_ku\|_{\tomega_{(k)}}\|\tpb_{(k),s}\rho^2_k\tpb^{*}_{(k),s}N^{q+1}_k\tpb_{(k),s}\chi_ku\|_{\tomega_{(k)}}, \notag
 \end{align*}
where $\Tilde{\rho}_k\in \mathscr{C}^{\infty}_c(\Cn)$ is another cut-off function such that $\supp \Tilde{\rho}_k\supset \supp \rho_k$ and $\supp \Tilde{\rho}_k \cap \supp\chi_k=\emptyset$. By direct computation,
\begin{align*}
\tpb_{(k),s}\rho^2_k\tpb^{*}_{(k),s}N^{q+1}_k\tpb_{(k),s}\chi_k u&=(\pb \rho^2_k)\wedge \tpb^{*}_{(k),s}N^{q+1}_k\tpb_{(k),s}\chi_ku+ \rho^2_k\tpb_{(k),s}\tpb^{*}_{(k),s}N^{q+1}_k\tpb_{(k),s}\chi_k u\\&=(\pb \rho^2_k)\wedge \tpb^{*}_{(k),s}N^{q+1}_k\tpb_{(k),s}\chi_ku+ \rho^2_k\tpb_{(k),s}\chi_k u \notag\\&= (\pb \rho^2_k)\wedge \tpb^{*}_{(k),s}N^{q+1}_k\tpb_{(k),s}\chi_ku,\notag\end{align*}
where the second equality is from (\ref{4.4 eq 2}) and the third is by the fact that $\supp\rho_k \cap \supp \chi_k=\emptyset$.
We apply this computation to continue the previous estimate and get
\begin{align}\label{4.5 lem1 (3)}
\|\rho_k\tpb^{*}_{(k),s}N^{q+1}_k\tpb_{(k),s}\chi_k u\|^2_{\tomega_{(k)}}&\leq\|\Tilde{\rho}_k N^{q+1}_k\tpb_{(k),s}\chi_ku\|_{\tomega_{(k)}}\|(\pb \rho^2_k)\wedge \tpb^{*}_{(k),s}N^{q+1}_k\tpb_{(k),s}\chi_ku\|_{\tomega_{(k)}}\\&\lesssim (C_k)^{-\epsilon}\|\Tilde{\rho}_kN^{q+1}_k\tpb_{(k),s}\chi_k u\|_{\tomega_{(k)}}\|\Tilde{\rho}_k\tpb^*_{(k),s}N^{q+1}_{k}\tpb_{(k),s}\chi_k u\|_{\tomega_{(k)}},\notag 
\end{align} 
where the term $(C_k)^{-\epsilon}$ arises during the computation of $\pb\rho_k$. Moreover, the sequence $\Tilde{\rho}_k$ can be taken to satisfy the condition $\sup_{|\alpha|=1} |\p^{\alpha}\Tilde{\rho}_k|\lesssim (C_k)^{-\epsilon}$ since $\supp \Tilde{\rho}\subset \{z\in\Cn;2/7C_k^\epsilon<|z|<C_k^\epsilon\}$.  To conduct an iteration process, we need to show the following claim:

\begin{claim}
   There exists $\Tilde{\Tilde{\rho}}_k\in \mathscr{C}^{\infty}_c(\Cn)$ with $\supp \Tilde{\Tilde{\rho}}_k\supset \supp \Tilde{\rho}_k$ and $\supp \Tilde{\Tilde{\rho}}_k \cap \supp\chi_k=\emptyset$ such that $\sup_{|\alpha|=1} |\p^{\alpha}\Tilde{\Tilde{\rho}}_k|\lesssim (C_k)^{-\epsilon}$ and \begin{multline*}
   \|\Tilde{\rho}_kN^{q+1}_{k}\tpb_{(k),s}\chi_k u\|_{\tomega_{(k)}}\|\Tilde{\rho}_k\tpb^*_{(k),s}N^{q+1}_{k}\tpb_{(k),s}\chi_k u\|_{\tomega_{(k)}} \\ \lesssim (C_k)^{-\epsilon}\|\Tilde{\Tilde{\rho}}_kN^{q+1}_k\tpb_{(k),s}\chi_k u\|_{\tomega_{(k)}}\|\Tilde{\Tilde{\rho}}_k\tpb^*_{(k),s}N^{q+1}_{k}\tpb_{(k),s}\chi_k u\|_{\tomega_{(k)}}.    
   \end{multline*}
    
\end{claim}
To show the claim, by Lemma \ref{4.4 lem 1}, we get 
\[\|\Tilde{\rho}_kN^{q+1}_k\tpb_{(k),s}\chi_k u\|_{\tomega_{(k)}}\lesssim \|\tpb_{(k),s}\Tilde{\rho}_kN^{q+1}_{k}\tpb_{(k),s}\chi_k u\|_{\tomega_{(k)}}+\|\tpb^*_{(k),s}\Tilde{\rho}_kN^{q+1}_{k}\tpb_{(k),s}\chi_k u\|_{\tomega_{(k)}}.\]
Moreover, we compute directly that
\begin{align*}
\tpb_{(k),s}\Tilde{\rho}_kN^{q+1}_k\tpb_{(k),s}\chi_k u&=(\pb \Tilde{\rho}_k)\wedge N^{q+1}_k\tpb_{(k),s}\chi_k u+\Tilde{\rho}_k\tpb_{(k),s}N^{q+1}_k\tpb_{(k),s}\chi_k u =(\pb \Tilde{\rho}_k)\wedge N^{q+1}_k\tpb_{(k),s}\chi_k u;\\
\tpb^*_{(k),s}\Tilde{\rho}_kN^{q+1}_k\tpb_{(k),s}\chi_k u&=-\sum_{i=1}^n(\dfrac{\p \Tilde{\rho}_k}{\p z^i}d\zb^i\wedge_{\tomega_{(k)}})^* N^{q+1}_k\tpb_{(k),s}\chi_k u+\Tilde{\rho}_k\tpb^*_{(k),s}N^{q+1}_k\tpb_{(k),s}\chi_k u.
\end{align*}Substitute these equations into the estimate and then dominate $\|\Tilde{\rho}_kN^{q+1}_k\tpb_{(k),s}\chi_k u\|_{\tomega_{(k)}}$ by   \begin{align*}
 \|(\pb \Tilde{\rho}_k)\wedge N^{q+1}_k\tpb_{(k),s}\chi_k u\|_{\tomega_{(k)}}&+ \|\sum_{i=1}^n(\dfrac{\p \Tilde{\rho}_k}{\p z^i}d\zb^i\wedge_{\tomega_{(k)}})^* N^{q+1}_k\tpb_{(k),s}\chi_k u\|_{\tomega_{(k)}}+\|\Tilde{\rho}_k\tpb^*_{(k),s}N^{q+1}_k\tpb_{(k),s}\chi_k u\|_{\tomega_{(k)}}\\&\lesssim (C_k)^{-\epsilon}\|\Tilde{\Tilde{\rho}}_k N^{q+1}_k\tpb_{(k),s}\chi_k u\|_{\tomega_{(k)}}+\|\Tilde{\Tilde{\rho}}_k\tpb^*_{(k),s}N^{q+1}_k\tpb_{(k),s}\chi_k u\|_{\tomega_{(k)}},    
\end{align*}for some $\Tilde{\Tilde{\rho}}_k$ as described above. So, \begin{multline*}
\|\Tilde{\rho}_kN^{q+1}_{k}\tpb_{(k),s}\chi_k u\|_{\tomega_{(k)}}\|\Tilde{\rho}_k\tpb^*_{(k),s}N^{q+1}_{k}\tpb_{(k),s}\chi_k u\|_{\tomega_{(k)}}\\ \lesssim (C_k)^{-\epsilon}\|\Tilde{\Tilde{\rho}}_k N^{q+1}_k\tpb_{(k),s}\chi_k u\|_{\tomega_{(k)}}\|\Tilde{\Tilde{\rho}}_k\tpb^*_{(k),s}N^{q+1}_{k}\tpb_{(k),s}\chi_k u\|_{\tomega_{(k)}} +\|\Tilde{\Tilde{\rho}}_k\tpb^*_{(k),s}N^{q+1}_k\tpb_{(k),s}\chi_k u\|^2_{\tomega_{(k)}}.
\end{multline*}
For the last term of the right-hand side, we replace the $\rho_k$ by $\Tilde{\Tilde{\rho}}_k$ in (\ref{4.5 lem1 (3)}) and get
 \begin{equation*}
\|\Tilde{\Tilde{\rho}}_k\tpb^*_{(k),s}N^{q+1}_k\tpb_{(k),s}\chi_k u\|^2_{\tomega_{(k)}} \lesssim (C_k)^{-\epsilon}\|\Tilde{\Tilde{\Tilde{\rho}}}_kN^{q+1}_k\tpb_{(k),s}\chi_k u\|_{\tomega_{(k)}}\|\Tilde{\Tilde{\Tilde{\rho}}}_k\tpb^*_{(k),s}N^{q+1}_{k}\tpb_{(k),s}\chi_k u\|_{\tomega_{(k)}}.
 \end{equation*}Combining the above estimates, we have completed the claim. Next, by (\ref{4.5 lem1 (3)}) and  iterating the claim, we can conclude that for any integer $d\in\N$, there exists a constant $C$ and $\Tilde{\rho}_k\in \mathscr{C}^{\infty}_c(\Cn)$ with $\supp \Tilde{\rho}_k\supset \supp \rho_k$ and $\supp \Tilde{\rho}_k \cap \supp\chi_k=\emptyset$ such that
 \[
\|\rho_k\tpb^{*}_{(k),s}N^{q+1}_k\tpb_{(k),s}\chi_k u\|^2_{\tomega_{(k)}}\lesssim (C_k)^{-d}\|\Tilde{\rho}_kN^{q+1}_{k}\tpb_{(k),s}\chi_k u\|_{\tomega_{(k)}}\|\Tilde{\rho}_k\tpb^*_{(k),s}N^{q+1}_{k}\tpb_{(k),s}\chi_k u\|_{\tomega_{(k)}}.\]
 Finally, we need to show the following fact:
 \begin{claim}
For all $v\in \Omega^{0,q}_c(\Cn)$
\begin{equation*}
\|\tpb^*_{(k),s}N_k^{q+1}\tpb_{(k),s}v\|_{\tomega_{(k)}}\leq \|v\|_{\tomega_{(k)}} \quad;\quad \|N_k^{q+1}\tpb_{(k),s}v\|_{\tomega_{(k)}}\lesssim \|v\|_{\tomega_{(k)}}.   
\end{equation*}
  \end{claim}For the first term, by (\ref{4.4 eq 2}), we compute that
 \begin{align*}
\|\tpb^*_{(k),s}N_k^{q+1}\tpb_{(k),s}v\|^2_{\tomega_{(k)}}&=\left(N_k^{q+1}\tpb_{(k),s}v\mid \tpb_{(k),s}\tpb^*_{(k),s}N_k^{q+1}\tpb_{(k),s}v\right)_{\tomega_{(k)}}=\left(N_k^{q+1}\tpb_{(k),s}v\mid \tpb_{(k),s}v\right)_{\tomega_{(k)}}\\ &=\left(\tpb^*_{(k),s}N_k^{q+1}\tpb_{(k),s}v\mid v\right)_{\tomega_{(k)}}\leq\|\tpb^*_{(k),s}N_k^{q+1}\tpb_{(k),s}v\|_{\tomega_{(k)}}\|v\|_{\tomega_{(k)}}.\end{align*} We get $\|\tpb^*_{(k),s}N_k^{q+1}\tpb_{(k),s}v\|_{\tomega_{(k)}}\leq \|v\|_{\tomega_{(k)}}$.  
 The second term follows by Lemma \ref{4.4 lem 1} that
 \begin{align*}
\|N_k^{q+1}\tpb_{(k),s}v\|_{\tomega_{(k)}}&\lesssim\|\tpb_{(k),s}N_k^{q+1}\tpb_{(k),s}v\|_{\tomega_{(k)}}+\|\tpb^*_{(k),s}N_k^{q+1}\tpb_{(k),s}v\|_{\tomega_{(k)}}\\ &=\|\tpb^*_{(k),s}N_k^{q+1}\tpb_{(k),s}v\|_{\tomega_{(k)}}\leq\|v\|_{\tomega_{(k)}},    
 \end{align*}since $\tpb_{(k),s}N_k^{q+1}\tpb_{(k),s}=0$. We completed the proof of the second claim. After combining all the above results, we know that for any integer $d\in\N$, there exists a constant $C$ such that
 \begin{equation}\label{040766}
\|\rho_k\tpb^{*}_{(k),s}N^{q+1}_k\tpb_{(k),s}\chi_k u\|_{\tomega_{(k)}}\leq C(C_k)^{-d}\|u\|_{\tomega_{(k)}}.    
 \end{equation}
Symmetrically, we can literally repeat the process to show the analogous statement:
\begin{equation}\label{0407329}
\|\rho_k\tpb_{(k),s}N^{q+1}_k\tpb^{*}_{(k),s}\chi_ku\|_{\tomega_{(k)}} \lesssim C(C_k)^{-d}\|u\|_{\tomega_{(k)}}.
\end{equation}
Then the lemma follows by (\ref{4.5 lem1 (0)}), (\ref{040766}), (\ref{0407329}) and a density argument. 
 \end{proof}
\begin{cor}\label{4.5 thm2 cor}
    For all $d\in\N$, \[\|\tpb^{*}_{(k),s}u_{(k)}\|_{\tomega_{(k)}}^2+\|\tpb_{(k),s}u_{(k)}\|_{\tomega_{(k)}}^2=O((C_k)^{-d}), \quad \text{as }\, k\To\infty. 
    \]
\end{cor}
\begin{proof}
    Recall the fact that $\tpb_{(k),s}\Tilde{B}^{q}_{(k),s}=0$ and $\tpb^*_{(k),s}\Tilde{B}^{q}_{(k),s}=0$. 
    \begin{align*}
    \tpb_{(k),s}u_{(k)}&=(\pb \Tilde{\chi}_k)\wedge\Tilde{B}^{q}_{(k),s} \chi_ku+\Tilde{\chi}_k\tpb_{(k),s}\Tilde{B}^{q}_{(k),s} \chi_ku=(\pb \Tilde{\chi}_k)\wedge\Tilde{B}^{q}_{(k),s} \chi_ku;\\
\tpb^*_{(k),s}u_{(k)}&=-\sum_{i=1}^{n}\dfrac{\p \Tilde{\chi}_k}{\p z^i}(d\zb^i)\wedge^*_{\tomega_{(k)}}\Tilde{B}^{q}_{(k),s} \chi_ku+\Tilde{\chi}_k \ \tpb^*_{(k),s}\Tilde{B}^{q}_{(k),s} \chi_ku=-\sum_{i=1}^{n}(\dfrac{\p \Tilde{\chi}_k}{\p z^i}(d\zb^i)\wedge_{\tomega_{(k)}}^*)\Tilde{B}^{q}_{(k),s} \chi_ku.   
    \end{align*}Observe that derivatives of $\Tilde{\chi}_k$ are supported in the annuli $\{3(C_k)^{\epsilon}/7<|z|<6(C_k)^{\epsilon}/7\}$ and $\rho_k\equiv 1$ on the annuli. We can see \begin{align*}
\|\tpb_{(k),s}u_{(k)}\|^2_{\tomega_{(k)}} \lesssim \|\rho_k\Tilde{B}^{q}_{(k),s}\chi_k u\|^2_{\tomega_{(k)}} \quad;\quad 
\|\tpb^*_{(k),s}u_{(k)}\|^2_{\tomega_{(k)}} \lesssim \|\rho_k\Tilde{B}^{q}_{(k),s}\chi_k u\|^2_{\tomega_{(k)}}.
\end{align*} By Lemma \ref{4.5 thm2 lem1}, we can immediately derive the corollary.
\end{proof}
\begin{thm}\label{4.5 thm2}

\begin{equation*}
\|P^q_{(k),s}u_{(k)}-u_{(k)}\|_{\omega_{(k)}} \rightarrow 0, \quad \text{as }\,k\To\infty.
\end{equation*}
\end{thm}
\begin{proof}
   Define $
    u_{k}(z):=(C_k)^{n/2}u_{(k)}(\sqrt{C_k}z)\in\Omega^{0,q}_c(B(1))\subset \Omega^{0,q}_c(M)$. Then we have
    \[
    \|P^q_{(k),s}u_{(k)}-u_{(k)}\|_{\omega_{(k)},B(\sqrt{C_k})}=\|P^q_{(k),s} u_k-u_k\|_{\omega,B(1)}.
    \]By (\ref{sca i3}) and (\ref{Hohohoho}),
    \begin{align*}
     \|P^q_{k,s} u_k-u_k\|^2_{\omega,B(1)}&=\|P^q_{k}e^{\phi_k} u_k\otimes s_k-e^{\phi_k} u_k\otimes s_k\|^2_{\omega,B(1)}\\ &\lesssim (C_k)^d\left(\Box^{q}_{k}e^{\phi_k}u_k\otimes s_k \mid e^{\phi_k}u_k\otimes s_k\right)_{\omega}\\&=(C_k)^d\left(\|\pb^{q,*}_{k,s}u_k\|^2_{\omega}+\|\pb^q_{k,s}u_k\|^2_{\omega}\right)\\&= (C_k)^{(d+1)}\left(\|\pb^{q,*}_{(k),s}u_{(k)}\|^2_{\omega}+\|\pb^q_{(k),s}u_{(k)}\|^2_{\omega}\right)\To 0.
    \end{align*}
    The last equality is by (\ref{sca pb}), and the last convergence is from  Corollary \ref{4.5 thm2 cor}.
   \end{proof}
Before overcoming the Statement \ref{4.2 statement 2} for the case $p\in M(q)$, we need another Lemma:
\begin{lem}\label{4.3 lem 3}
 For any $v\in \Omega^{0,q}_c(\Cn)$,
 \[
 \left(v \mid P^q_{(k),s}\chi_ku -P^q_{s} u\right)_{\omega_0} \rightarrow 0, \quad \text{as }\,k\To\infty.
 \] 
\end{lem}
\begin{proof}
    Let $v\in\Omega^{0,q}_c(\Cn)$. For any fixed positive integer $n_0\in\N$, observe that for large enough $k$, we can estimate that 
    \begin{multline*}
      \aaa{\left(v \mid P^q_{(k),s}\chi_ku -P^q_s u\right)_{\omega_0}}\leq \aaa{\left(v\mid (P^q_{(k),s}\chi_k-P^q_{s})\chi_{n_0} u\right)_{\omega_0}} \\+ \|v\|_{\omega_0}\|(P^q_{(k),s}\chi_k-P^q_s)(\chi_{n_0}-1)u\|_{\omega_0}.
      \end{multline*}
      Moreover, by (\ref{uni 1}) and (\ref{uni 3}), we see $P^q_{(k),s}\chi_k-P^q_s$ are uniformly bounded linear map on the space $L^2_{\omega_0}(\Cn,\T)$. Given an arbitrary number $\varepsilon>0$, we can fix $n_0$ large enough such that
    \[
    \|v\|_{\omega_0}\|(P^q_{(k),s}\chi_k-P^q_s)(\chi_{n_0}-1)u\|_{\omega_0}<\varepsilon/2 \quad \text{for all}\,k\in\N.
    \]
    Furthermore, by the assumption that $P^q_{(k),s}(z,w)\To P^q_{0,s}(z,w)$ locally uniformly, 
    \[
    \aaa{\left(v\mid (P^q_{(k),s}\chi_k-P^q_s)\chi_{n_0} u\right)_{\omega_0}} \rightarrow 0 \quad \text{as }\, k\To\infty.
    \] Finally, combining the estimates above, we obtain $\aaa{\left(v |P^q_{(k),s}\chi_k u-P^q_su\right)_{\omega_0}}<\varepsilon$ for large enough $k$.
\end{proof}
We are now ready to complete the proof of the Statemant \ref{4.2 statement 2} for the case $p\in M(q)$.

\begin{proof}[Proof of Statemant \ref{4.2 statement 2} for $p\in M(q)$]
   By Theorem \ref{4.1 Model theorem}, we may assume that $u$ is of the form $u=z^{\alpha}_qe^{-\sum|\lambda_i||z^i|^2}d\zb^I$ for some $\alpha\in\N_0^n$ by density argument. By (\ref{uni 1}), (\ref{uni 3}) and Lemma \ref{4.5 thm 1} and the decrease of $u$,  \begin{align}\label{4.5 thm Bu=u (1)}
       \| P^q_{(k),s} (\chi_k u-u_{(k)})\|_{\omega_0} \lesssim \|\chi_k u-u_{(k)}\|_{\omega_{0}}\leq\|\chi_k u-u\|_{\omega_0}+\|u-u_{(k)}\|_{\omega_{0}}\To 0.
   \end{align}
   To show $P^q_{s} u=u$, let $v\in\Omega^{0,q}_{c}(\Cn)$ and observe that
    \begin{multline*}
       \left(v\mid P^q_{s} u-u\right)_{\omega_0} =\left(v\mid P^q_{s} u-P^q_{(k),s} \chi_ku\right)_{\omega_0}+\left(v\mid P^q_{(k),s}(\chi_ku-u_{(k)})\right)_{\omega_0}\\+ \left(v\mid P^q_{(k),s} u_{(k)}-u_{(k)}\right)_{\omega_0}+\left(v\mid u_{(k)}-u\right)_{\omega_0}.  
    \end{multline*}
    By Lemma \ref{4.5 thm 1}, Theorem \ref{4.5 thm2}, Lemma \ref{4.3 lem 3} and (\ref{4.5 thm Bu=u (1)}), the right-hand side of the above equation must tend to zero.  
\end{proof}
\begin{rmk}
 In the function case $q=0$, we may obtain the same result by the process in \cite[Section 4.3]{me} under \textbf{spectral gap conditions of a suitable exponential rate} (cf. \cite[Def. 1.3]{me}) and replace $k$ by $C_k$.    
\end{rmk}
\subsection{Heat kernel proof for Bergman kernel asymptotic}\label{section 3.4}
In this section, we provide an idea to establish a simpler proof of main theorems (Theorem \ref{main theorem}, Theorem \ref{main theorem 2})
under a stronger spectral gap condition as follows:

\begin{assumption}[Global large spectral gap condition ]\label{spectral gap 2}Denote the global spectral gap $c'_k$ by
\[
c'_k:=\inf \left(\text{spec }\Box_k^q -\{0\}\right)
\] and assume it satifies the \textbf{large condition}: \[\liminf\frac{c_k^{'}}{C_k}>0.\]
\end{assumption} 
As we have already proven main theorems in the previous sections, we will not go into all the details, especially in the asymptotic of Heat kernels. Instead, we will focus on the application of Heat kernels to Bergman kernels.

Define the Heat operator $H^q_k(t):L^2_{\omega,\phi_k}(M,\TM\otimes L_k)\To L^2_{\omega,\phi_k}(M,\TM\otimes L_k)$ which is the functional calculus of $e^{-st}$ with respect to $\Box^q_k$. Define the Heat kernel $H^q_k(t,z,w)$ as the Schwartz kernel of $H^q_k(t)$. Recall the construction in Chapter \ref{Chapter 2}, we now consider the case \[ A^q_k(z,w)=H^q_k(\frac{t}{C_k},z,w). \]
To apply Theorem \ref{thm uniform bound}, we compute that\[
N_{m,k}=\sup_{s\in [0,\infty)}(\frac{s}{C_k})^me^{-\frac{s}{C_k}t}\lesssim t^{-m}.
\]Moreover, the operator norm of $H^q_k(\frac{t}{C_k},z,w)$ is less than or equal to one. All estimates smoothly depend on the parameter $t$. Hence, we have the following corollary:   
\begin{cor}(The local uniform bounds for Heat kernels)
In the localization process introduced in Section \ref{Aaaac}, the scaled Heat kernels \[C_k^{-n}H^{q}_{k,s}(t/C_k,z/\sqrt{C_k},w/\sqrt{C_k})\] are locally uniformly bounded in the $\mathscr{C}^{\infty}$-topology on $\R^{+}\times \Cn\times\Cn$.
\end{cor}
Similarly to the idea in the Bergman kernel case, we assume that the scaled Heat kernel $H^{q}_{(k),s}(t/C_k,z,w)$ converges to a kernel section $H^{q,s}_{0}(t,z,w)$ in the $\mathscr{C}^{\infty}$-topology. Next, we may follow the limiting process as Chen presented in \cite[Section 3.4]{jt} and get the result that $H^{q}_{0,s}(t,z,w)$ must be the Heat kernel in the model case of $\Cn$ equipped with the weight function $\phi_0$ and the standard Hermitian form $\omega_0$. We conclude the following theorem without proof. 
\begin{thm}\cite[Section 3.4]{jt}\label{Heat} Denote $H^q_0(t,z,w)$ as the Heat kernel with respect to the Kodaira Laplacian $\Box^q_0$ on $\Cn$ as considered in Theorem \ref{4.1 Model theorem}. Then the scaled heat kernels $H^{q}_{(k),s}(t/C_k,z,w)$ converge to $H^q_0(t,z,w)$ in $\mathscr{C}^{\infty}$-topology on $\R^+\times \Cn\times\Cn$.
\end{thm}  

Next, we set an operator $A^q_k(t):L^2_{\omega,\phi_k}(M,\TM\otimes L_k)\To L^2_{\omega,\phi_k}(M,\TM\otimes L_k)$ by \[
A^q_k(t):=H^q_k(t/C_k)-P^q_k.
\]
Then, $A^q_k(t)$ is the functional calculus of \[a^q_k(t,s):=\left(1-\mathbbm{1}_{[0,c'_k]}(s)\right)e^{-\frac{t}{C_k}s},\]with respect to $\Box^q_k$. Here, $c_k^{'}$ is the global spectral gap (cf. Assumption \ref{spectral gap 2}). Let $A^q_k(t,z,w)$ be the Schwartz kernel of $A^q_k(t)$. By Theorem \ref{thm uniform bound}, the $\mathscr{C}^\ell$-norm for some $\ell\in\N$ is locally dominated by 
\[
\sup_{s\in [0,\infty)}|a^q_k(t,s)|+\sup_{s\in[0,\infty)}\left[(\frac{s}{C_k})^m+(\frac{s}{C_k})^{2m}\right]a^q_k(t,s)\leq \left(e^{-\frac{c'_k}{C_k}t}+t^{-m}+t^{-2m}\right)\lesssim t^{-N(\ell)},
\]for some $N(\ell)\in\N$ depending on $\ell$. Here, we use the Assumption \ref{spectral gap 2}. We deduce that for any $\ell\in\N$ and bounded domain $U\subset\Cn$, there exists $N(\ell,U)\in\N$ such that
\[
\aaa{A^q_{(k),s}(t,x,y)}_{\mathscr{C}^{\ell}(U\times U)}\lesssim t^{-N(\ell,U)}.
\]Finally, we derive that\[
\lim_{t\To\infty}\limsup_{k\To\infty}\aaa{A^q_{(k),s}(x,y)}_{\mathscr{C}^\ell(U\times U)}=0.
\]Observe that\[
H^{q}_{(k),s}(t,z,w)=P^q_{(k),s}(z,w)+A^q_{(k),s}(z,w).
\] 
Hence,
\[
\lim_{t\To\infty}\limsup_{k\To\infty}\aaa{H^{q}_{(k),s}(t,z,w)-P^q_{(k),s}(z,w)}_{\mathscr{C}^\ell(U\times U)}=0.
\] By Theorem \ref{Heat}, we have
\begin{align*}
0&=\lim_{t\To\infty}\lim_{k\To\infty}\aaa{H^{q}_{(k),s}(t,z,w)-P^q_{(k),s}(z,w)}_{\mathscr{C}^\ell(U\times U)}\\&=\lim_{t\To\infty} \aaa{H^q_{0,s}(t,z,w)- \lim_{k\To \infty}P^q_{(k),s}(z,w)}_{\mathscr{C}^\ell(U\times U)}.
\end{align*} This tells us $\lim_{t\To\infty}H^q_{0,s}(t,z,w)=\lim_{k\To \infty}P^q_{(k),s}(z,w)$ in $\mathscr{C}^{\infty}$-topology. Here, we assume $\lim_{k\To \infty}P^q_{(k),s}(z,w)$ exists. To prove the main theorems for Bergman kernel, it is sufficient to show the following Lemma:
\begin{lem} We have the convergence
 \[\lim_{t\To\infty}H^q_{0,s}(t,z,w)=P^q_{0,s}(z,w),
    \]local uniformly in $\mathscr{C}^{\infty}$-topology on $\Cn\times \Cn$.
\end{lem}
\begin{proof}[Sketch of the proof]
First, we need a spectral gap for the model case. We set $M=\Cn$ with the standard Hermitian form $\omega=\omega_0$, and consider the sequence trivial line bundle $L_k$ over $\Cn$ with the quadratic weight function $\phi_k=C_k\phi_0$. We apply the results in Section \ref{section 4.3} to the case set above. Thus, $\Tilde{\phi}_{(k)}=\phi_0$ and $\Tilde{\omega}_{(k)}=\omega_0$. If the curvature is non-degenerate, that means $\lambda_i\neq 0$ for all $i=1,\cdots,n$. By Theorem \ref{4.4 thm spectral gap}, we have the global spectral gap on $\Cn$ as follows:
\begin{equation}\label{Hi}
 c:=\inf(\text{spec}\Box^q_0-\{0\})>0.   
\end{equation}
Moreover, if the curvature is degenerate, that means $\lambda_i=0$ for some $i=1,\cdots,n$. We can also obtain (\ref{Hi}) by (\ref{4.4 lem 1 (2)}) and (\ref{123444}) in the proof of Lemma \ref{4.4 lem 1}. 

Now, we have the global spectral gap for the model case on $\Cn$ with the weight function $\phi_0$ and Hermitian form $\omega_0$. To show the lemma, we repeat the proof of Lemma \ref{lemma mapping property} and Theorem \ref{thm uniform bound}. Define \[A_{0}^q(t,x,y):=H^q_{0}(t,z,w)-P^q_{0}(x,y)\] which is the Schwartz kernel of the functional calculus of $a^q_{0}$ where \[a^q_{0}:=(1-\mathbbm{1}_{[0,c]}(s))e^{-ts}.\] As introduced in Section \ref{Aaaac}, we consider the localized kernels \[
A_{0,s}^q(t,x,y):=H^q_{0,s}(t,z,w)-P^q_{0,s}(x,y).
\] Our goal is to show that $A_{0,s}^q(t,x,y)\To 0$ in $\mathscr{C}^{\infty}$-topology as $t\To \infty$. Let $\chi$ and $\rho$ be two cut-off functions in $\Cn$ and $\ell\in\N$. By the process in the proof of Lemma \ref{lemma mapping property} and Theorem \ref{thm uniform bound}, we can see that \[
\aaa{\chi(x)\left(A^q_{0,s}(t,x,y)\right)
\rho(y)}_{\mathscr{C}^\ell(\Cn\times \Cn)}\] is dominated by (up to constants depending on $\chi$,$\rho$, and $\ell$) the operator norms of \[
\|A^q_{0,s}\|, \quad \|\Box^q_{0,s}A^q_{0,s}\|,\quad\|\Box^q_{0,s}A^q_{0,s}\Box^q_{0,s}\| \quad \text{on }\, L^2_{\omega_0}(\Cn,\T).
\]The operator norms above can be bounded by \[e^{-tc},\quad \sup_{s\geq c} s e^{-ts},\quad \sup_{s\geq c} s^2 e^{-ts},\] respectively. Since all quantities are $O(t^{-N})$ for some $N\in\N$, they tend to zero uniformly as $t$ goes to infinity.

\end{proof}

%% file: main.bbl
\begin{thebibliography}{99}











\bibitem{001}
Robert Berman, Bo Berndtsson, Johannes Sjöstrand. A direct approach to Bergman kernel asymptotics for positive
line bundles, Ark. Math. 46 (2008), no. 2,197–217.

\bibitem{002}
Thierry Bouche. Convergence de la métrique de Fubini-Study d'un fibré linéaire positif. Annales de l'Institut Fourier, Volume 40 (1990) no. 1, pp. 117-130. doi : 10.5802/aif.1206. http://www.numdam.org/articles/10.5802/aif.1206/

\bibitem{003}
David Catlin. The Bergman kernel and a theorem of Tian. In Analysis and geometry
in several complex variables (Katata, 1997), Trends Math., pages 1-23. Birkhäuser
Boston, Boston, MA, 1999.

\bibitem{jt}
Eric Jian-Ting Chen. Semi-classical Heat Kernel Asymptotics and Morse Inequalities, Academia Sinica, 18(2023) no.4, 365-418.


\bibitem{me}
Yueh-Lin Chiang. Semi-classical Asymptotics of Bergman and Spectral Kernels for $(0,q)$-forms, Bulletin of the Institute of Mathematics, Academia Sinica, 18(2023) no.3, 299-364.
\bibitem{01}
Dan Coman, Wen Lu, Xiaonan Ma, George Marinescu.
Bergman kernels and equidistribution for sequences of line bundles on Kähler manifolds,
Advances in Mathematics,
Volume 414,
2023,
108854,
ISSN 0001-8708.
\bibitem{004}
Xianzhe Dai, Kefeng Liu, Xiaonan Ma. On the asymptotic expansion of Bergman kernel, C. R. Math. Acad. Sci. Paris
339 (2004), no. 3, 193–198.
\bibitem{Davies}
Edward Davies. Spectral theory and differential operators, volume 42 of Cambridge Studies in Advanced Mathematics. Cambridge University Press, Cambridge, 1995.
\bibitem{005}
Xianzhe Dai, Kefeng Liu, Xiaonan Ma. On the asymptotic expansion of Bergman kernel, J. Differential Geom., 72,
(2006), no. 1, 1–41.
\bibitem{02}
Jean-Pierre Demailly. Complex Analytic and Differential Geometry, Open Content Book, 2012.
\bibitem{De}
Jean-Pierre Demailly. Champs magnétiques et inégalités de Morse pour la $d^{\prime \prime }$-cohomologie. Annales de l'Institut Fourier, Volume 35 (1985) no. 4, pp. 189-229. 
\bibitem{Hou}
 Yu-Chi Hou. Asymptotic expansion of the Bergman kernel via semiclassical symbolic calculus. Bulletin of the Institute of Mathematics, Academia Sinica, (N.S.) 17(1), 1–51
(2022).
\bibitem{006}
Chin-Yu Hsiao, George Marinescu. Asymptotics of spectral function of lower energy
forms and Bergman kernel of semi-positive and big line bundles, Comm. Anal. Geom.,
22 (2014), No.1, 1-108.

\bibitem{Kohn}
Chin-Yu Hsiao, Weixia Zhu.  Heat kernel asymptotics for Kohn Laplacians on CR manifolds. Journal of Functional Analysis(2023), 284(2), 109755. 
\bibitem{M}
Xiaonan Ma, George Marinescu. Holomorphic Morse inequalities and Bergman
kernels, volume 254 of Progress in Mathematics, Birkhäuser Verlag, Basel, 2007.

\bibitem{007}
Xiaonan Ma, George Marinescu. The first coefficients of the asymptotic expansion
of the Bergman kernel of the $\text{spin}^{c}$ Dirac operator, Internat. J. Math., 17 (2006),
No.6, 737-759.
\bibitem{Gaffney}
Matthew Gaffney. Hilbert space methods in the theory of harmonic integrals. Trans. Amer.
Math. Soc., 78:426–444, 1955.


\bibitem{Yos}
Kosaku Yosida. Functional analysis. Classics in Mathematics, Springer-Verlag, Berlin, 1995.
Reprint of the sixth (1980) edition.
\end{thebibliography}
